\documentclass{amsart}
\usepackage{amsthm,amsfonts,amssymb,amsmath,amscd,enumitem,
}
\usepackage{textcomp}
\usepackage{comment} 
\usepackage{xcolor}
\usepackage{longtable}
\usepackage{url}
\usepackage[all,2cell, cmtip]{xy}
\xyoption{2cell}
\includecomment{comment}
\specialcomment{AndreNotes}{\begingroup\sffamily\color{blue}}{\endgroup}
\specialcomment{AmitNotes}{\begingroup\sffamily\color{red}}{\endgroup}
 \begin{document}

\title[Picard groupoids and $\Gamma$-categories]{Picard groupoids and $\Gamma$-categories}
\author[A. Sharma]{Amit Sharma}

\email{asharm24@kent.edu}
\address {Department of Mathematical sciences\\ Kent State University\\
  Kent, OH}

\date{Dec. 14, 2019}
%

\newcommand{\CONT}{\noindent}
\newcommand{\FIG}{Fig.\ }
\newcommand{\FIGS}{Figs.\ }
\newcommand{\SEC}{Sec.\ }
\newcommand{\SECS}{Secs.\ }
\newcommand{\TAB}{Table }
\newcommand{\TABS}{Tables }
\newcommand{\EQ}{Eq.\ }
\newcommand{\EQS}{Eqs.\ }
\newcommand{\APP}{Appendix }
\newcommand{\APPS}{Appendices }
\newcommand{\CHP}{Chapter }
\newcommand{\CHPS}{Chapters }

\newcommand{\OFF}{\emph{G2off}~}
\newcommand{\TOO}{\emph{G2Too}~}
\newcommand{\CatS}{Cat_{\bigS}}
\newcommand{\PCat}{\mathbf{Perm}}
\newcommand{\PCatOL}{\mathbf{Perm}_{\textit{OL}}}
\newcommand{\PCatSM}{\mathbf{Perm}_{\otimes}}
\newcommand{\mdlPCat}{\left( \PCat, \mathbf{E} \right)}
\newcommand{\PicS}{{\underline{\pic}}^{\oplus}}
\newcommand{\HPicS}{{Hom^{\oplus}_{\pic}}}
\newcommand{\HPerm}{{Hom^{\oplus}_{\PCat}}}
\newcommand{\HopL}{{Hom^{\text{un-opLax}}_{\PCat}}}

\newtheorem{thm}{Theorem}[section]
\newtheorem{lem}[thm]{Lemma}
\newtheorem{conj}[thm]{Conjecture}
\newtheorem{coro}[thm]{Corollary}
\newtheorem{prop}[thm]{Proposition}

\theoremstyle{definition}
\newtheorem{df}[thm]{Definition}
\newtheorem{nota}[thm]{Notation}

\newtheorem{ex}[thm]{Example}
\newtheorem{exs}[thm]{Examples}

\theoremstyle{remark}
\newtheorem*{note}{Note}
\newtheorem*{rem}{Remark}
\newtheorem{ack}{Acknowledgments}

\newcommand{\ChI}{{\textit{\v C}}\textit{ech}}
\newcommand{\Ch}{{\v C}ech}

\newcommand{\ChZG}{hermitian line $0$-gerbe}
\renewcommand{\theack}{$\! \! \!$}

\newcommand{\Leins}{\mathfrak{L}}
\newcommand{\ChG}{flat hermitian line $1$-gerbe}
\newcommand{\ChC}{hermitian line $1$-cocycle}
\newcommand{\ChGG}{flat hermitian line $2$-gerbe}
\newcommand{\ChCC}{hermitian line $2$-cocycle}
\newcommand{\id}{id}
\newcommand{\SigA}{\Sigma^\A}
\newcommand{\totSigInf}[1]{\mathbb{L}\Sigma^{\infty}{#1}}
\newcommand{\SigAvec}[1]{\Sigma^\A_{\length{#1}}}
\newcommand{\length}[1]{\mid \vec{#1} \mid}
\newcommand{\underMap}[1]{\mid #1 \mid}
\newcommand{\bProd}{\text{box product}}
\newcommand{\EPs}{\E^{\textit{Ps}}}
\newcommand{\ePs}{\epsilon^{\textit{Ps}}}
\newcommand{\EStr}{\E_{\textit{str}}}
\newcommand{\Pnor}{\Lbb}
\newcommand{\elR}{\textit{el}^{R}(\gn{n}){\mid}_{\N}}
\newcommand{\elRGen}[1]{\textit{el}^{R}(\gn{#1}){\mid}_{\N}}
\newcommand{\elRGeneral}[1]{\textit{el}^{R}(#1){\mid}_{\N}}
\newcommand{\elGen}[2]{\textit{el}^{#1}{#2}{\mid}_{\N}}
\newcommand{\elKbar}{\textit{el}^{\Kbar}(\gn{n}){\mid}_{\N}}
\newcommand{\unit}[1]{\mathrm{1}_{#1}}
\newcommand{\bike}{\text{bicycle}}
\newcommand{\bikes}{\text{bicycles}}
\newcommand{\psbike}{\text{pseudo bicycle}}
\newcommand{\psbikes}{\text{pseudo bicycles}}
\newcommand{\PsBikes}[2]{\textbf{Bikes}^{\textit{Ps}}(#1,#2)}
\newcommand{\strbike}{\text{strict bicycle}}
\newcommand{\strbikes}{\text{strict bicycles}}
\newcommand{\StrBikes}[2]{\textbf{Bikes}^{\textit{Str}}(#1,#2)}
\newcommand{\Bikes}[2]{\textbf{Bikes}(#1,#2)}
\newcommand{\NLC}[2]{\textbf{NorLaxCones}(#1,#2)}
\newcommand{\LCns}[2]{\textbf{LaxCones}(#1,#2)}
\newcommand{\LC}{\mathfrak{C}}
\newcommand{\Coker}{Coker}
\newcommand{\Com}{Com}
\newcommand{\Hom}{Hom}
\newcommand{\Mor}{Mor}
\newcommand{\Map}{Map}
\newcommand{\alg}{alg}
\newcommand{\an}{an}
\newcommand{\Ker}{Ker}
\newcommand{\Ob}{Ob}
\newcommand{\SymMon}{\mathbf{SymMon}}
\newcommand{\Kbar}{\overline{\K}}
\newcommand{\KSeg}{\K}
\newcommand{\KSegLax}{{\Kbb}}
\newcommand{\Proj}{\mathbf{Proj}}
\newcommand{\topo}{\mathbf{Top}}
\newcommand{\inrt}{\mathbf{Inrt}}
\newcommand{\act}{\mathbf{Act}}
\newcommand{\n}{\underline{n}}
\newcommand{\pn}{\underline{n}^+}
\newcommand{\kan}{\mathcal{K}}
\newcommand{\pkan}{\mathcal{K}_\bullet}
\newcommand{\Kan}{\mathbf{Kan}}
\newcommand{\pKan}{\mathbf{Kan}_\bullet}
\newcommand{\gp}{\mathcal{A}_\infty}
\newcommand{\mdl}{\mathcal{M}\textit{odel}}
\newcommand{\sSets}{\mathbf{sSets}}
\newcommand{\sSetsQ}{(\mathbf{sSets, Q})}
\newcommand{\sSetsK}{(\mathbf{sSets, \Kan})}
\newcommand{\pSSets}{\mathbf{sSets}_\bullet}
\newcommand{\pSSetsK}{(\mathbf{sSets}_\bullet, \Kan)}
\newcommand{\pSSetsQ}{(\mathbf{sSets_\bullet, Q})}
\newcommand{\cyl}{\mathbf{Cyl}}
\newcommand{\lin}{\mathcal{L}_\infty}
\newcommand{\Vect}{\mathbf{Vect}}
\newcommand{\Aut}{Aut}
\newcommand{\Ein}{E_\infty}
\newcommand{\EinS}{E_\infty{\text{- space}}}
\newcommand{\EinSs}{E_\infty{\text{- spaces}}}
\newcommand{\EinC}{\text{coherently commutative monoidal category}}
\newcommand{\EinCs}{\text{coherently commutative monoidal categories}}
\newcommand{\EinLO}{E_\infty{\text{- local object}}}
\newcommand{\EinSLO}{\E_\infty\S{\text{- local object}}}
\newcommand{\pic}{\mathcal{P}\textit{ic}}
\newcommand{\Dlin}{\pic}
\newcommand{\Gmn}{\Gamma \left(m_n \right)}
\newcommand{\bigS}{\mathbf{S}}
\newcommand{\bigA}{\mathbf{A}}
\newcommand{\bhom}{\mathbf{hom}}
\newcommand{\bHom}[3]{\mathbf{hom}_{#3}(#1, #2)}
\newcommand{\bhomK}{\mathbf{hom}({\textit{K}}^+,\textit{-})}
\newcommand{\Bhom}{\mathbf{Hom}}
\newcommand{\bhomk}{\mathbf{hom}^{{\textit{k}}^+}}
\newcommand{\CatHom}[3]{[#1,#2]^{#3}}
\newcommand{\pCatHom}[3]{[#1,#2]_\bullet^{#3}}
\newcommand{\pHomCat}[2]{[#1,#2]_{\bullet}}
\newcommand{\Dlino}{\pic^{\textit{op}}}
\newcommand{\lino}{\mathcal{L}^{\textit{op}}_\infty}
\newcommand{\lind}{\mathcal{L}^\delta_\infty}
\newcommand{\linK}{\mathcal{L}_\infty(\kan)}
\newcommand{\linC}{\mathcal{L}_\infty\text{-category}}
\newcommand{\linCs}{\mathcal{L}_\infty\text{-categories}}
\newcommand{\ainCs}{\text{additive} \ \infty-\text{categories}}
\newcommand{\ainC}{\text{additive} \ \infty-\text{category}}
\newcommand{\inC}{\infty\text{-category}}
\newcommand{\inCs}{\infty\text{-categories}}
\newcommand{\gS}{{\Gamma}\text{-space}}
\newcommand{\ggS}{\Gamma \times \Gamma\text{-space}}
\newcommand{\gSs}{\Gamma\text{-spaces}}
\newcommand{\ggSs}{\Gamma \times \Gamma\text{-spaces}}
\newcommand{\gO}{\Gamma-\text{object}}
\newcommand{\gSCat}{{\Gamma}\text{-space category}}
\newcommand{\gCat}{{\Gamma}\text{- category}}
\newcommand{\gCats}{{\Gamma}\text{- categories}}
\newcommand{\gGpd}{{\Gamma}\text{- groupoid}}
\newcommand{\gGpds}{{\Gamma}\text{- groupoids}}
\newcommand{\gCAT}{{\Gamma}\Cat}
\newcommand{\gCATop}{\gCAT^{\textit{op}}}
\newcommand{\OplaxExp}[2]{\left({#2}^{#1}\right)^{\textit{OL}}}
\newcommand{\OplaxNatExp}[2]{\left({#2}^{#1}\right)^{\textit{Ps}}}
\newcommand{\SMExp}[2]{\left(#2^{\Leins(#1)}\right)^{\textit{OL}}}
\newcommand{\SMNatExp}[2]{\left(#2^{\Leins(#1)}\right)^{\textit{Ps}}}
\newcommand{\OplaxSec}[2]{{\Gamma}^{\textit{OL}} \left(\N, \OplaxExp{#1}{#2}\right)}
\newcommand{\SMSec}[2]{{\Gamma}_{\otimes}^{\textit{str}}\left(\Leins, \OplaxExp{\Leins(#1)}{#2}\right)}
\newcommand{\SMHom}[2]{[#1,#2]^\otimes}
\newcommand{\SMHomNor}[2]{[#1,#2]_\bullet^\otimes}
\newcommand{\OLHom}[2]{[#1,#2]^{\textit{OL}}}
\newcommand{\OLHomNor}[2]{[#1,#2]^{\textit{OL}}}
\newcommand{\StrSMHom}[2]{[#1,#2]_\otimes^{\textit{str}}}
\newcommand{\parMHom}[2]{\mathbf{SBikes}(#1,#2)}
\newcommand{\StrSec}[2]{{\Gamma}^{Str}(#1,#2)}
\newcommand{\StrExp}[2]{\left(#2^{A{#1}}\right)^{\textit{Str}}}
\newcommand{\Odot}[2]{\underset{#1=1}{\overset{#2} \odot}}
\newcommand{\Prod}[2]{\underset{#1=1}{\overset{#2} \prod}}
\newcommand{\Otimes}[2]{\underset{#1=1}{\overset{#2} \otimes}}
\newcommand{\OtimesC}[3]{\underset{#1=1}{\overset{#2} \underset{#3} \otimes}}
\newcommand{\pss}{\mathbf{S}_\bullet}
\newcommand{\gSC}{{{{\Gamma}}\mathcal{S}}}
\newcommand{\ggSC}{{\Gamma\Gamma\mathcal{S}}}
\newcommand{\gSD}{\mathbf{D}(\gSC^{\textit{f}})}
\newcommand{\sCat}{\mathbf{sCat}}
\newcommand{\pSCat}{\mathbf{sCat}_\bullet}
\newcommand{\Dhom}{\mathbf{R}Hom_{\pic}}
\newcommand{\gop}{\Gamma^{\textit{op}}}
\newcommand{\gn}[1]{\Gamma^{#1}}
\newcommand{\gnk}[2]{\Gamma^{#1}({#2}^+)}
\newcommand{\gnf}[2]{\Gamma^{#1}({#2})}
\newcommand{\ggn}[1]{\Gamma\Gamma^{#1}}
\newcommand{\fU}{\mathbf{U}}
\newcommand{\cDN}{\underset{\mathbf{D}[\textit{n}^+]}{\circ}}
\newcommand{\cDK}{\underset{\mathbf{D}[\textit{k}^+]}{\circ}}
\newcommand{\cDL}{\underset{\mathbf{D}[\textit{l}^+]}{\circ}}
\newcommand{\cD}{\underset{\gSD}{\circ}}
\newcommand{\cDT}{\underset{\gSD}{\widetilde{\circ}}}
\newcommand{\ppsSets}{\sSets_{\bullet, \bullet}}
\newcommand{\gdHom}{\underline{Hom}_{\gSD}}
\newcommand{\HomU}{\underline{Hom}}
\newcommand{\ominf}{\Omega^\infty}
\newcommand{\ev}{ev}
\newcommand{\PNat}{\overline{\L}}
\newcommand{\PStr}{\L}
\newcommand{\PStrA}{\P^{\textit{str}}_\A}
\newcommand{\PNatnor}{\P^{\textit{Nat}}_{\textit{nor}}}
\newcommand{\PStrnor}{\P^{\textit{str}}_{\textit{nor}}}
\newcommand{\cu}{C(X;\mathfrak{U}_I)}
\newcommand{\Sing}{Sing}
\newcommand{\gSR}{{\Gamma}\Gamma \mathcal{S}}
\newcommand{\gSRP}{{\Gamma}\Gamma \mathcal{S}^{\textit{proj}}}
\newcommand{\tensPGSR}[2]{#1 \underset{\gSR}\wedge #2}
\newcommand{\pTensP}[3]{#1 \underset{#3}\wedge #2}
\newcommand{\MGCat}[2]{\underline{\map}_{\gCAT}({#1},{ #2})}
\newcommand{\MGBoxCat}[2]{\underline{\map}_{\gCAT}^{\Box}({#1},{ #2})}
\newcommand{\TensPFunc}[1]{- \underset{#1} \otimes -}
\newcommand{\TensP}[3]{#1 \underset{#3}\otimes #2}
\newcommand{\odotPFunc}[1]{- \underset{#1} \odot -}
\newcommand{\odotP}[3]{#1 \underset{#3}\odot #2}
\newcommand{\EgSRP}{\Ein\gSRP}
\newcommand{\Catop}{\Cat^{\textit{op}}}
\newcommand{\Catl}{\Cat_{\textit{l}}}
\newcommand{\SQCup}[2]{\underset{#1 = 1}{\overset{#2} \sqcup}}
\newcommand{\Ind}{\textit{Ind}}

\newcommand{\liminj}{\varinjlim}
\newcommand{\limproj}{\varprojlim}
\newcommand{\Lbb}{\overline{\overline{\L}}}
\newcommand{\Kbb}{\overline{\overline{\K}}}
\newcommand{\Nat}{\mathbb{N}}
\newcommand{\PLbb}{\P\overline{\overline{\L}}}
\newcommand{\gpd}{\mathbf{Gpd}}
\newcommand{\PGpd}{\mathbf{PGpd}}
\newcommand{\ugrph}{\mathbf{Graph_0}}
\newcommand{\MdlCatG}{(\mathbf{\Cat}, \mathbf{\gpd})}
\newcommand{\MdlPCatG}{(\mathbf{\PCat}, \mathbf{\gpd})}
\newcommand{\ud}[1]{\underline{#1}}
\newcommand{\MdlPCatP}{(\mathbf{\PCat}, \mathbf{Pic})} 
\newcommand{\PicH}{(\mathbf{Pic}, \textit{Str})}
\newcommand{\CCPicH}{(\gCAT^{\textit{f}}, \textit{Str})}
\newcommand{\SOTPicH}{(\pGSC^{\textit{f}}[1], \textit{Str})}
\newcommand{\pGSC}{{{{\Gamma}}\mathcal{S}}_\bullet}
\newcommand{\norN}[1]{N^{\textit{nor}}(#1)}
\newcommand{\unorT}[1]{\tau^{\textit{un}}(#1)}

\def\Pic{\mathbf{2}\mathcal P\textit{ic}}
\def\nc{\mathbb C}

\def\Z{\mathbb Z}
\def\P{\mathcal P}
\def\J{\mathcal J}
\def\I{\mathcal I}
\def\nC{\mathbb C}
\def\H{\mathcal H}
\def\A{\mathcal A}
\def\C{\mathcal C}
\def\D{\mathcal D}
\def\E{\mathcal E}
\def\G{\mathcal G}
\def\B{\mathcal B}
\def\L{\mathcal L}
\def\U{\mathcal U}
\def\K{\mathcal K}
\def\N{\mathcal N}
\def\M{\mathcal M}
\def\O{\mathcal O}
\def\R{\mathcal R}
\def\S{\mathcal S}
\def\F{\mathcal F}

\newcommand{\undertilde}[1]{\underset{\sim}{#1}}
\newcommand{\abs}[1]{{\lvert#1\rvert}}
\newcommand{\mC}[1]{\mathfrak{C}(#1)}
\newcommand{\sigInf}[1]{\Sigma^{\infty}{#1}}
\newcommand{\x}[4]{\underset{#1, #2}{ \overset{#3, #4} \prod }}
\newcommand{\mA}[2]{\textit{Add}^n_{#1, #2}}
\newcommand{\mAK}[2]{\textit{Add}^k_{#1, #2}}
\newcommand{\mAL}[2]{\textit{Add}^l_{#1, #2}}
\newcommand{\Mdl}[2]{\L_\infty}
\newcommand{\inv}[1]{#1^{-1}}
\newcommand{\Lan}[2]{\mathbf{Lan}_{#1}(#2)}
\newcommand{\Bike}[3]{{#1}:{#2} \leadsto {#3}}
\newcommand{\POb}[1]{\mathbf{P}(#1)}
\newcommand{\SuppBike}[1]{{#1}_{\textit{Supp}}}
\newcommand{\Supp}[1]{{\textit{Supp}(#1)}}
\newcommand{\nBoxProd}[2]{\underset{i=1}{\overset{#1} \boxtimes} #2}
\newcommand{\BoxProd}[1]{\overset{#1} \boxtimes}
\newcommand{\nOdotProd}[2]{\underset{i=1}{\overset{#1} \odot} #2}
\newcommand{\OdotBike}[3]{\SuppBike{#1}(#2) \odot \SuppBike{#1}(#3)}
\newcommand{\partition}[2]{\delta^{#1}_{#2}}
\newcommand{\inclusion}[2]{\iota^{#1}_{#2}}

\newcommand{\del}{\partial}
\newcommand{\sCatO}{\mathcal{S}Cat_\O}
\newcommand{\FCgop}{\mathbf{F}\mC{N(\gop)}}
\newcommand{\hProd}{{\overset{h} \oplus}}
\newcommand{\hProdn}{\underset{n}{\overset{h} \oplus}}
\newcommand{\hProdk}[1]{\underset{#1}{\overset{h} \oplus}}
\newcommand{\map}{\mathcal{M}\textit{ap}}
\newcommand{\MapC}[3]{\mathcal{M}\textit{ap}_{#3}(#1, #2)}
\newcommand{\MGS}[2]{\underline{\map}_{\gSC}({#1},{ #2})}
\newcommand{\MGSR}[2]{\underline{\map}_{\gSR}({#1},{ #2})}
\newcommand{\MGSBox}[2]{\underline{\map}^{\Box}_{\gSC}({#1},{ #2})}
\newcommand{\Aqcat}[1]{\underline{#1}^\oplus}
\newcommand{\Cat}{\mathbf{Cat}}
\newcommand{\pCat}{\mathbf{CAT}_\bullet}
\newcommand{\SMCat}{\mathbf{SMCat}}
\newcommand{\Sp}{\mathbf{Sp}}
\newcommand{\SpStb}{\mathbf{Sp}^{\textit{stable}}}
\newcommand{\SpStr}{\mathbf{Sp}^{\textit{strict}}}
\newcommand{\Sspec}{\mathbb{S}}
\newcommand{\pC}[1]{{#1}_\bullet}

\newcommand{\gM}{\Gamma \mathcal{M}}
\newcommand{\gMC}[1]{\Gamma \mathcal{#1}}
\newcommand{\gMP}{\gM^{\textit{proj}}}
\newcommand{\tensPGSC}[2]{#1 \underset{\gSC}\wedge #2}
\newcommand{\tensPGM}[2]{#1 \underset{\gM}\wedge #2}
\newcommand{\EgM}{\Ein\gM}
\newcommand{\EgSC}{\Ein\gSC}
\newcommand{\EgMP}{\Ein\gMP}

\begin{abstract}
 In this paper we construct a symmetric monoidal closed model category of coherently commutative Picard groupoids. We construct another model category structure on the category of (small) permutative categories whose fibrant objects are (permutative) Picard groupoids.
 The main result is that the Segal's nerve functor induces a Quillen equivalence between the two aforementioned model categories.
Our main result implies the classical result that Picard groupoids model stable homotopy one-types.
\end{abstract}

\maketitle

\tableofcontents

 \section[Introduction]{Introduction}
\label{Introduction}
\begin{sloppypar}
Picard groupoids are interesting objects both in topology and algebra. A major reason for interest in topology is because they classify stable homotopy $1$-types which is a classical result appearing in various parts of the literature \cite{JO}\cite{Pat12}\cite{GK11}. The category of Picard groupoids is the archetype example of a $2$-Abelian category, see \cite{DP08}. A theory of $2$-chain complexes of Picard
groupoids was developed in \cite{dr-mm-v}. A simplicial cohomology with coefficients in Picard groupoids was introduced in the paper \cite{carrasco-martinez-moreno}. This cohomology was used in \cite{SV} to construct a TQFT called the \emph{Dijkgraaf-Witten} theory. This (Picard) groupoidification of cohomology played a vital role in explaining a mysterious \emph{integration theory} introduced in \cite{fq} which is pivotal in constructing the aforementioned TQFT functor. A model category structure on (permutative) Picard groupoids was constructed in \cite{GM97}.
\end{sloppypar}
A tensor product of Picard groupoids was defined in \cite{Schmitt2}. However, a shortcoming of the category of Picard groupoids remains: unlike the category of abelian groups, it is not a symmetric monoidal closed category. In this paper we address this problem by proposing another model for Picard groupoids based on $\gCats$. A $\gCat$ is a functor from the (skeletal) category of finite based sets $\gop$ into the category of all (small) categories $\Cat$. We denote the category of all $\gCats$ and natural transformations between them by $\gCAT$. Along the lines of the construction of the stable Q-model category in \cite{Schwede} we construct a symmetric monoidal closed model category structure on $\gCAT$ which we refer to as the model category structure of \emph{coherently commutative Picard groupoids}. A $\gCat$ $X$ is called a \emph{coherently commutative Picard groupoid} if it satifies the Segal condition, see \cite{segal} and moreover it has homotopy inverses. These $\gCats$ are fibrant objects in our model category of coherently commutative Picard groupoids. The main objective of this paper is to compare the category of all (small) Picard groupoids with our model category of coherently commutative Picard groupoids. There are many variants of the category of Picard groupoids. All of these variant categories are \emph{fibration categories} but they do not have a model category structure. Due to this  shortcoming, in this paper we will work with permutative Picard groupoids, namely symmetric monoidal groupoids which are strictly associative and unital and each object is invertible upto isomorphism. We construct another model category structure on the category of all (small) permutative categories and strict symmetric monoidal functors $\PCat$ whose fibrant objects are (permutative) Picard groupoids. This model category is denoted by $\MdlPCatP$ and called the model category of Picard groupoids. The category of (permutative) Picard groupoids is a reflective subcategory of $\PCat$ and the inclusion functor is a right Quillen functor of a Quillen equivelance between $\MdlPCatP$ and the model category constructed in \cite{GM97}. The main result of this paper is that the classical Segal's nerve functor induces a Quillen equivalence between the model category of coherently commutative Picard groupoids and the model category of Picard groupoids. The left adjoint to the classical Segal's nerve does not have a simple description therefore we first show that the \emph{thickened} Segal's nerve functor \cite{Sharma} induces a Quillen equivalence between the aforementioned model categories. There is a natural equivalence between thickened Segal's nerve functor and the classical Segal's nerve functor which now implies the desired result. Our main result implies that the homotopy category of $\MdlPCatP$ is equivalent to the the full subcategory of the homotopy category of the stable $Q$-model category structure on (normalized) $\gSs$, constructed in \cite{Schwede}, whose objects are $\gSs$ having upto two non-zero homotopy groups only in degree $0$ or $1$. Thus our main result implies the classical result that Picard groupoids model stable homotopy one-types.

We establish another Quillen equivalence between a second pair of model category structures on the same two underlying categories. We first construct another cartesian closed (combinatorial) model category structure on $\Cat$, denoted by $\MdlCatG$, whose fibrant objects are groupoids. 
We then transfer this model category structure on the category of all permutative categories $\PCat$. The fibrant objects in this model category are permutative groupoids. This model category a simplicial model category and it is denoted by $\MdlPCatG$. Since the model category $\MdlCatG$ is combinatorial, it induces a \emph{projective} model category structure on $\gCAT$. Along the lines of the construction of the model category of coherently commutative monoidal categories in \cite{Sharma}, we localize this projective model category to get a symmetric monoidal closed model category structure on $\gCAT$. The fibrant objects of this model category can be described as coherently commutative monoidal groupoids. These model categories are instrumental in the construction of the model categories featuring in our main result. We go on to show that the classical Segal's nerve functor is the right Quillen functor of a Quillen equivalence between the aforementioned model categories.

  \section{The Setup}
In this section we will collect the machinery needed for various constructions in this paper. We begin with a review of permutative categories.
We will also give a quick review of $\gCats$ and collect some useful results about them. We will also construct a cartesian closed (simplicial) model category structure on the category of (small) categories $\Cat$ which will be used throughout this paper.
\subsection[Review of Permutative categories]{Review of Permutative categories}
\label{Preliminaries}
 In this subsection we will briefly review the theory of permutative categories and monoidal and oplax functors between them. The definitions reviewed here and the notation specified here will be used throughout this paper.
 
 \begin{df}
 	A symmetric monoidal category $C$ is called a \emph{permutative} category or a \emph{strict} symmetric monoidal category if it's monoidal structure is strictly associative and unital.
 	\end{df}

 \begin{df}
 \label{oplax-sym-mon-functor}
 An  \emph{oplax symmetric monoidal} functor $F$ is a triple $(F, \lambda_F, \epsilon_F)$, where $F:C \to D$ is a functor  between symmetric monoidal categories $C$ and $D$, 
  \begin{equation*}
  \label{oplax-nat-trans}
  \lambda_F:F \circ (- \underset{C} \otimes -) \Rightarrow (- \underset{D} \otimes -) \circ (F \times F)
  \end{equation*}
  is a natural transformation 
  and $\epsilon_F:F(\unit{C}) \to \unit{D}$ is a morphism in $D$, such that the following
  three conditions are satisfied
  \begin{enumerate}[label = {OL.\arabic*}, ref={OL.\arabic*}]
 \item \label{OpL-unit} For each objects $c \in Ob(C)$, the following diagram commutes
 \[
  \xymatrix@C=12mm{
 F(\unit{C} \underset{C} \otimes c) \ar[r]^{\lambda_F(\unit{C}, c) \ \ \ }  \ar[d]_{F(\beta_C(c))}  
 &F(\unit{C}) \underset{D} \otimes F(c) \ar[d]^{\epsilon_F \underset{D} \otimes id_{F(c)}} \\
 F(c) \ar[r]_{\inv{\beta_D(F(c))} \ \ \ \ } & \unit{D} \underset{D} \otimes F(c)
 }
 \]
 \item \label{OpL-symmetry} For each pair of objects $c_1, c_2 \in Ob(C)$, the following diagram commutes
 \[
  \xymatrix@C=22mm{
 F(c_1\underset{C} \otimes c_2) \ar[r]^{\lambda_F(c_1, c_2) \ \ \ }  \ar[d]_{F(\gamma_C(c_1,c_2))}  
 &F(c_1) \underset{D} \otimes F(c_2) \ar[d]^{\gamma_D(F(c_1), F(c_2)) \ } \\
 F(c_2 \underset{C} \otimes c_1) \ar[r]_{\lambda_F(c_2, c_1) } & F(c_2) \underset{D} \otimes F(c_1)
 }
 \] 
\item \label{OpL-associativity} For each triple of objects $c_1, c_2, c_3 \in Ob(C)$, the following diagram commutes
 \[
  \xymatrix@C=10mm{
  & F(c_1\underset{C} \otimes c_2) \underset{D} \otimes F(c_3) \ar[rd]^{\lambda_F(c_1, c_2) \underset{D} \otimes id_{F(c_3)}} \\
 F((c_1 \underset{C} \otimes c_2) \underset{C} \otimes c_3) \ar[ru]^{\lambda_F(c_1\underset{C} \otimes c_2, c_3) \ \ \ }  \ar[d]_{F(\alpha_C(c_1,c_2, c_3))}  
 &&(F(c_1) \underset{D} \otimes F(c_2)) \underset{D} \otimes F(c_3) \ar[d]^{\alpha_D(F(c_1), F(c_2),  F(c_3)) \ } \\
 F(c_1 \underset{C} \otimes (c_2 \underset{C} \otimes c_3)) \ar[rd]_{\lambda_F(c_1, c_2 \underset{C} \otimes c_3) \ \ \ \ } && F(c_1) \underset{D} \otimes (F(c_2) \underset{D} \otimes F(c_3)) \\
 & F(c_1) \underset{D} \otimes (F(c_2 \underset{C} \otimes c_3)) \ar[ru]_{ \ \  \ \ \ \ \ id_{F(c_1)} \underset{D} \otimes  \lambda_F(c_2, c_3)}
 }
 \]

\end{enumerate}

 \end{df}
 \begin{nota}
 \label{unital-SM-Func}
 We will say that a functor $F:C \to D$ between two symmetric monoidal categories is \emph{unital} or \emph{normalized} if it preserves the unit of the symmetric monoidal structure \emph{i.e.} $F(\unit{C}) = \unit{D}$. In particular, 
 we will say that an oplax symmetric monoidal functor is a unital (or normalized) oplax symmetric monoidal functor if the morphism $\epsilon_F$ is the identity.
\end{nota}

\begin{df}
	An oplax symmetric monoidal functor $F = (F, \lambda_F, \epsilon_F)$ is called a \emph{strong} symmetric monoidal functor (or just a symmetric monoidal functor) if $\lambda_F$ is a natural isomorphism and $\epsilon_F$ is also an isomorphism.
\end{df}

\begin{df}
	An oplax symmetric monoidal functor $F = (F, \lambda_F, \epsilon_F)$ is called a \emph{strict} symmetric monoidal functor if it is unital and $\lambda_F$ is the identity natural transformation.
	\end{df}

\begin{df}
 We define a category $\Catl$ whose objects
are pairs $(C, c)$, where $C$ is a category and $c:\ast \to C$ is a functor whose value is
$c \in C$. A morphism from $(C, c)$ to $(D, d)$ in $\Catl$ is a pair $(F, \alpha)$, where
$F:C \to D$ is a functor and $\alpha:F(c) \to d$ is a map in $D$. The category $\Catl$ is
equipped with an obvious projection functor
\begin{equation}
\label{univ-proj-cat}
p_l:\Catl \to \Cat.
\end{equation}
We will refer to the functor $p_l$ as the \emph{universal left fibration over} $\Cat$.
\end{df}
Let $(F, \alpha):(C, c) \to (D, d)$ and $(G, \beta):(D, d) \to (E, e)$ be  A pair of composable arrows in $\Catl$. Then their composite is defined as follows:
\[
(G, \beta) \circ (F, \alpha) := (G \circ F, \beta \cdot (id_G \circ \alpha)),
\]
where $\cdot$ represents vertical composition and $\circ$ represents horizontal
composition of 2-arrows in $\Cat$.
\begin{df}
The \emph{category of elements} of a $\Cat$ valued functor $F:C \to \Cat$, denoted by
$\int^{c \in C} F(c)$ or $\textit{el} F$, is a category which is defined by the following pullback square in $\Cat$:
\begin{equation*}
\xymatrix{
\int^{c \in C} F(c) \ar[r]^{p_2} \ar[d]_{p_1}& \Catl \ar[d]^{p_l} \\
C \ar[r]_{F \ \ \  } & \Cat
 }
\end{equation*}
\end{df}
The category $\int^{c \in C} F(c)$ has the following description:

The object set of $\int^{c \in C} F(c)$ consists of all pairs $(c, d)$, where $c \in Ob(C)$ and $d :\ast \to F(c)$ is a functor.
A map $\phi:(c_1, d_1) \to (c_2, d_2)$ is a pair $(f, \alpha)$, where $f:c_1 \to c_2$ is a
map in $C$ and $\alpha:F(f) \circ d_1 \Rightarrow d_2$ is a natural transformation. The category of elements of $F$
is equipped with an obvious projection functor $p:\int^{c \in C} F \to C$.
\begin{rem}
\label{simp-des-cat-el}
We observe that a functor $d:\ast \to F(c)$ is the same as an object $d \in F(c)$.
Similarly a natural transformation $\alpha:F(f) \circ d \Rightarrow b$ is the same as an
arrow $\alpha:F(f)(d) \to b$ in $F(a)$, where $f:c \to a$ is an arrow in $C$. This observation leads to
a simpler equivalent description of $\int^{c \in C} F(c)$. The objects of $\int^{c \in C} F(c)$ are
pairs $(c, d)$, where $c \in C$ and $d \in F(c)$. A map from $(c, d)$ to $(a, b)$ in $\int^{c \in C} F(c)$
is a pair $(f, \alpha)$, where $f:c \to a$ is an arrow in $C$ and $\alpha:F(f)(d) \to b$ is an arrow in $F(a)$.
\end{rem}

\subsection[Review of $\gCats$]{Review of $\gCats$}
\label{Rev-gamma-cats}
In this subsection we will briefly review the theory of $\gCats$. We begin by introducing some notations which will be used throughout the paper.
\begin{nota}
We will denote by $\ud{n}$ the finite set $\lbrace 1, 2, \dots, n \rbrace$ and by $n^+$ the based set $\lbrace 0, 1, 2, \dots, n \rbrace$ whose basepoint is the element $0$.
\end{nota}
\begin{nota}
 We will denote by $\N$ the skeletal category of finite unbased sets whose objects are $\ud{n}$ for all $n \ge 0$ and maps are functions of unbased sets. The category $\N$ is a (strict) symmetric monoidal category whose symmetric monoidal structure will be denoted by $+$. For to objects $\ud{k}, \ud{l} \in \N$ their \emph{tensor product} is defined as follows:
 \[
 \ud{k} + \ud{l} := \ud{k + l}.
 \]
\end{nota}
\begin{nota}
 We will denote by $\gop$ the skeletal category of finite based sets whose objects are $n^+$ for all $n \ge 0$ and maps are functions of based sets.
\end{nota}

\begin{nota}
 We denote by $\inrt$ the subcategory of $\gop$ having the same set of objects as $\gop$
 and intert morphisms.
\end{nota}
\begin{nota}
 We denote by $\act$ the subcategory of $\gop$ having the same set of objects as $\gop$
 and active morphisms.
\end{nota}
\begin{nota}
A map $f:\ud{n} \to \ud{m}$ in the category $\N$ uniquely determines an active map in $\gop$ which we will denote by $f^+:n^+ \to m^+$.
This map agrees with $f$ on non-zero elements of $n^+$.
\end{nota}
\begin{nota}
 Given a morphism $f:n^+ \to m^+$ in $\gop$, we denote by $\text{Supp}(f)$ the largest
 subset of $\n$ whose image under $f$ does not caontain the basepoint of $m^+$.
 The set $\text{Supp}(f)$ inherits an order from $\n$ and therefore could be regarded as
 an object of $\N$. We denote by $\text{Supp}(f)^+$ the based set $\text{Supp}(f) \sqcup \lbrace 0 \rbrace$
 regarded as an object of $\gop$ with order inherited from $\n$.
\end{nota}

\begin{prop}
 Each morphism in $\gop$ can be uniquely factored into a composite of an inert map followed
 by an active map in $\gop$.
\end{prop}

\begin{proof}
Any map $f:n^+ \to m^+$ in the category $\gop$ can be factored as follows:
\begin{equation}
\label{fact-in-gamma-op}
  \xymatrix@C=11mm{
 n^+  \ar[rd]_{f_{\textit{inrt}}}   \ar[rr]^{f}  &&m^+   \\
  &\text{Supp}(f)^+  \ar[ru]_{f_{\textit{act}}}
 }
\end{equation}
where $\text{Supp}(f) \subseteq n$ is the \emph{support} of the function $f$
\emph{i.e.}  $\text{Supp}(f)$ is the largest subset of $n$ whose elements are mapped
by $f$ to a non zero element of $m^+$. The map $f_{inrt}$ is the projection of
$n^+$ onto the support of $f$ and therefore $f_{inrt}$ is an inert map. The map
$f_{act}$ is the restriction of $f$ to $\text{Supp}(f) \subset \n$,
therefore it is an \emph{active} map in $\gop$.

\end{proof}

\subsection[Natural model category structure on $\Cat$]{The model category structure of groupoids on $\Cat$}
\label{Nat-Mdl-Str-CAT}
In this subsection we will construct another model categort structure on on the category of all small categories $\Cat$ which we will refer to as the \emph{ model category structure of groupoids}.
 We will show that the weak equivalences in this model structure are those functors which induce a weak homotopy equivalence on their nerve. The model category structure is obtained by a \emph{left Bousfield localization} of the \emph{natural} model category structure on $\Cat$ with respect to the set of functors $\lbrace i:0 \to I, i \times I \rbrace$, where $I$ is the category $0 \to 1$ and $i(0) = 0$.
 \begin{prop}
 	\label{char-fib-mdl-gpd}
 	A category $C$ is local with respect to the set of functors $\lbrace i:0 \to I, i \times I \rbrace$ if and only if it is a groupoid.
 	\end{prop}
 \begin{proof}
 	By Lemma \cite[Lemma E.4]{Sharma}, $C$ is $\lbrace i:0 \to I, i \times I \rbrace$-local if and only if the functor
 	\[
 	[i, C]:[I, C] \to [0, C] \cong C
 	\]
 	is an equivalence of categories. We observe that the category $[I, C]$ is the arrow category of $C$ and the functor $[i, C]$ has a left adjoint $L:C \to [I, C]$ which is defined on objects by $L(c) := id_c$ and on morphisms in the obvious way. If the functor $[i, C]$ is an equivalence of categories then, for any map $f:c \to d$ in $C$, the counit map $L[i, C](f):id_c \to f$ which is the following commutative diagram:
 	\begin{equation*}
 	\xymatrix{
 	c \ar@{=}[d] \ar@{=}[r] & c \ar[d]^f \\
 	c \ar[r]_f & d
   }
 	\end{equation*}
 	is an isomorphism in $[I, C]$. This implies that $f:c \to d$ is invertible and hence $C$ is a groupoid.
 	Conversely, if $C$ is a groupoid then $[I, C] \cong [J, C]$ where $J$ is the groupoid $0 \cong 1$ and for any category $C$, the functor category $[J, C]$ is equivalent to $C$.
 	\end{proof}

 \begin{thm}
 	\label{Gpd-mdl-str}
 	There is a combinatorial model category structure of the category of (small) categories $\Cat$ in which a functor $F:A \to B$ is
 	\begin{enumerate}
 	\item  a cofibration if it is monic on objects.
 	\item a weak equivalence if the following functor
 	\[
 	[i, F]:[B, Z] \to [A, Z]
 	\]
 	is an equivalence of categories for each groupoid $Z$.
 	\item a fibration if it has the right lifting property with respect to functors which satisfy both $(1)$ and $(2)$. 
 	  \end{enumerate}
 	\end{thm}
 \begin{proof}
 We want to carry out a left Bousfield localization of the natural model category of (small) categories with respect to the set $\lbrace i:0 \to I, i \times I \rbrace$. The existence of this localization follows from theorem \ref{local-tool}. $(1)$ follows from the aforementioned theorem. $(2)$ follows from proposition \ref{char-fib-mdl-gpd} and \cite[Lemma E.4]{Sharma}. $(3)$ follows from the fact that fibrations in any model category are completely determined by cofibrations and weak equivalences.
 	
 \end{proof}

 \begin{nota}
 	We will refer to the above model category structure as the  model category structure of \emph{groupoids} on $\Cat$ and denote the model category by $\MdlCatG$.
 	We will refer to a fibration in this model category as a \emph{path fibration} of categories and refer to a weak equivalence as a \emph{groupoidal equivalence} of categories.
 	\end{nota}
 \begin{rem} 
 	Every category is cofibrant in the model category of groupoids. A category is fibrant if and only if it is a groupoid.
 	\end{rem}
 \begin{rem} 
 	A groupoidal equivalence between groupoids is an equivalence of categories.
 \end{rem}

 \begin{prop}
 	\label{Kan-fib-btw-gpds}
 	The nerve of a path fibration of categories between two groupoids is a Kan fibration of simplicial sets.
 	\end{prop}
 \begin{proof}
 	Let $p:C \to D$ be a path fibration of categories such that both $C$ and $D$ are groupoids. Since the inclusion functor $i:0 \to I$ is an acyclic fibration in the model category of groupoids therefore $p$ has the right lifting property with respect to $i$. Since both $C$ and $D$ are groupoids by assumption, this is equivalent to $p$ having the right lifting property with respect to $0 \hookrightarrow J$, where $J$ is the groupoid $0 \cong 1$. This implies that $p$ is an isofibration \emph{i.e.} a fibration in the natural model category structure on $\Cat$, see \cite{Sharma}. Now the nerve functor takes an isofibrations to a pseudo-fibration \emph{i.e.} a fibration in the Joyal model category on simplicial sets so $N(p):N(C) \to N(D)$ is a pseudo-fibration. However both $N(C)$ and $N(D)$ are Kan complexes and a pseudo-fibration between Kan complexes is a Kan fibration, see \cite[Cor. 4.28]{AJ2}.
 	\end{proof}
  Next we are interested in providing a characterization of weak equivalences and fibrations in the model category of groupoids. In doing so we will be using the following adjunction:
  \[
  \tau_1:\Cat \rightleftarrows \sSets:N
  \]
  \begin{lem}
  	\label{N-htpy-ref}
  	The adjunction $(\tau_1, N)$ is a Quillen adjunction between the model category of groupoids and the Kan model category of simplicial sets.
  	Further the adjunction is also a homotopy reflection.
  	\end{lem}
  \begin{proof}
  	In order to prove the first statement we will use \cite[Prop. E.2.14]{AJ1}. Clearly the functor $\tau_1$ takes cofibrations to cofibrations. The right adjoint functor $N$ takes fibrations between fibrant objects in the Kan model category of (small) categories to Kan fibrations of simplicial sets by proposition \ref{Kan-fib-btw-gpds}. Thus $(\tau_1, N)$ is a Quillen adjunction between the aforementioned model category structures.
  	
  	\end{proof}
  
  The inclusion functor $\gpd \to \Cat$, where $\gpd$ is the full subcategory whose objects are groupoids, has a left adjoint which we denote by $\Pi_1:\Cat \to \gpd$. We now describe the groupoid $\Pi_1(C)$. Let $G = U(C)$ denote the underlying graph of the category $C$. The arrows of $C$ are the edges of the graph $G$. For each non-invertible arrow $f:c_1 \to c_2$ of $C$ we add an edge $\inv{f}:c_2 \to c_1$ to $G$. The objects of $\Pi_1(C)$ would be the same as those of $C$. An arrow of $\Pi_1(C)$ is a $k$-tuple of composable, non-identity, arrows in $C$ namely
  $(f_1, f_2, \dots, f_k)$ where $\textit{dom}(f_{r}) = \textit{cod}(f_{r-1})$, for $1 < r \le k$. In addition, for each $c \in Ob(C)$ we have an arrow $(id_c)$. Further, no two adjacent arrows are inverses \emph{i.e.} $f_r \ne \inv{f_{r-1}}$, for $1 < r \le k$. The composition operation is given by concatenation.
  \begin{nota}
  	A morphism in $\Pi_1(C)$ is a tuple. We will refer to the size of this tuple as the \emph{length} of the morphism, namely the lenth of $f = (f_1, f_2, \dots, f_k)$ is $k$ and we denote the length of a morphism by $|f|$.
  	\end{nota}
  
  \begin{rem}
  	\label{QEq-mdl-gpd}
  	In the paper \cite{JT1} a model category structure was constructed on the full subcategory of $\Cat$ whose objects are groupoids $\gpd$. We will refer to this model category as the \emph{natural model category of groupoids}
  	The functor $\Pi_1$ is a left Quillen functor of a Quillen adjunction
  	\[
  	\Pi_1:\Cat \rightleftharpoons \gpd:i
  	\]
  	where $\Cat$ is endowed with the model category structure of groupoids and $\gpd$ is the natural model category of groupoids. This Quillen adjunction is a Quillen equivalence.
  \end{rem}

 The following proposition will be used repeatedly in this paper:
 \begin{prop}
 	\label{Pi1-pres-prod}
 	The free groupoid functor $\Pi_1:\Cat \to \gpd$ preserves products.
 \end{prop}
%
%
  
  \begin{prop}
  	\label{char-weak-eq}
  	A functor $F:C \to D$ is a groupoidal equivalence if and only if the induced functor $\Pi_1(F):\Pi_1(C) \to \Pi_1(D)$ is an equivalence of categories.
  	\end{prop}
  \begin{proof}
  	The unit of the adjunction $\Pi_1 \dashv i$ gives the following commutative diagram:
  	\begin{equation*}
  	\xymatrix{
  	C \ar[r] \ar[d]_F & \Pi_1(C) \ar[d]^{\Pi_1(F)} \\
  	D \ar[r]  & \Pi_1(D)
  }
  	\end{equation*}
  	where both vertical functors are inclusions. We will first prove that these two inclusion maps are both weak equivalences. Since $\Pi_1$ is a left adjoint to the inclusion functor $i$ therefore the inclusion functor $\iota_C:C \to \Pi_1(C)$ induces the following bijection for each groupoid $G$:
  	\[
  	\Cat(\Pi_1(C), G) \cong \Cat(C, G).
  	\]
  	Conside the following chain of bijections:
  	\begin{multline*}
  	\Cat(I, [\Pi_1(C), G]) \cong \Cat(I \times \Pi_1(C), G) \cong \Cat(\Pi_1(C), [I, G]) \\
  	 \cong \Cat(C, [I, G]) \cong \Cat(I \times C, G) \cong \Cat(I, [C, G]).
  	\end{multline*}
  The above two bijections together imply that we have the following equivalence of functor categories:
  \begin{equation*}
  [\iota_C, G]:[\Pi_1(C), G] \to [\Pi_1(C), G].
  \end{equation*}
  Now Theorem \ref{Gpd-mdl-str} $(2)$ implies that the two inclusion maps are weak equivalences in the model category structure of groupoids. Now the theorem follows from the two out of three property of weak equivalences in a model category.
  	\end{proof}
  
%
%
  Finally we would like to show that the groupoidal model category structure on $\Cat$ is cartesian closed.
  
  \begin{thm}
  	\label{cart-closed-gp-cat}
  	The groupoidal model category structure on $\Cat$ is cartesian closed.
  	\end{thm}
  \begin{proof}
  	Let $u:C \to D$ be a cofibration and $i:X \to Y$ be another cofibration in the model category of groupoids. We will like to show that their pushout product
  	\[
  u	\Box i: D \times X \underset{C \times X} \coprod C \times Y \to D \times Y
  	\]
  	is a cofibration which is acyclic if either $u$ or $i$ is acyclic in the model category of groupoids. The cofibrations in the groupoidal model structure are the same as the cofibrations in the natural model structure on $\Cat$ and so are the acyclic fibrations. Further every path fibration is a fibration in the natural model category structure. Now the cartesian closed natural model category structure implies that $u	\Box i$ is a cofibration. Let us now assume that $u$ is an acyclic cofibration. In order to show that $u \Box i$ is an acyclic cofibration it is sufficient to show that it has the left lifting property with respect to all path fibrations between groupoids, see \cite[]{Sharma}. Let $p:G \to H$ be such a path fibration. By adjointness, it is sufficient to show the existence of a lifting arrow $L$ whenever we have the following commutative diagram:
  	\begin{equation*}
  	\xymatrix{
  	C \ar[r] \ar[d]_u & [D, G] \ar[d]^{(i^\ast, p_\ast)} \\
  	D \ar[r] \ar@{-->}[ru]_L & [C, G] \underset{[C, H]} \times [D, H]
  }
  	\end{equation*}
  	Since $p:G \to H$ is a path fibration between groupoids therefore it is a fibration in the natural model category. The cartesian closed natural model category structure on $\Cat$ implies that the right vertical map is a fibration between groupoids and therefore it is a path fibration. Since $u$ is an acyclic cofibration in the groupoidal model category structure on $\Cat$ therefore it has the left lifting property with respect to all path fibrations. Thus we have shown that $u \Box i$ is an acyclic cofibration whenever $u$ is an acyclic fibration and $i$ is a fibration.
  	\end{proof}
  
  \begin{prop}
  	\label{simp-mdl-cat}
  	The model category of groupoids is a simplicial model category.
  	\end{prop}
  \begin{proof}
  	We will verify the conditions of \cite[Def. 4.1.12]{Hovey}. The bifunctor
  	\[
  	\MapC{-}{-}{\sSets}:\Cat^{op} \times \Cat \to \sSets
  	\]
  	is defined to be the composite:
  	\[
  	\Cat^{op} \times \Cat \overset{[-, -]} \to \Cat \overset{N} \to \sSets.
  	\]
  	This bifunctor is the left hom in the sense of the aforementioned definition. The right hom bifunctor
  	\[
  	\sSets^{op} \times \Cat \to \Cat
  	\]
  	is defined as follows:
  	\[
  	\sSets^{op} \times \Cat \to \Cat^{op} \overset{\tau_1 \times id} \times \Cat \overset{[-, -]} \to \Cat
  	\]
  	The \emph{tensor product} bifunctor
  	\[
  	\Cat \times \sSets \to \Cat.
  	\]
  	is defined by the following composite:
  	\[
  	\Cat \times \sSets \overset{id \times \tau_1} \to \Cat  \times \Cat \overset{- \times -} \to \Cat .
  	\]
  	Let $C$ and $D$ be two categories and let $S$ be a simplicil set. We have the following chain of bijections:
  	\[
  	\Cat(C, [\tau_1(S), D]) \cong \Cat(C \times \tau_1(S), D) \cong \Cat(\tau_1(S), [C, D]) \cong \sSets(S, N([C, D])).
  	\]
  	The above chain of bijections verifies \cite[Def. 4.1.12]{Hovey} and thus establishes the $2$-variable adjunction. Along the lines of the proof of the previous theorem, one can verify Lemma \ref{Q-bifunctor-char} $(3)$.
  	\end{proof}

  \section{Two model category structures on $\PCat$}
\label{perm-gpds}
We denote by $\PCat$ the category whose objects are \emph{permutative} categories namely symmetric monoidal categories which are strictly unital and strictly associative.
The morphisms of this category are \emph{strict} symmetric monoidal functors, namely those symmetric monoidal functors which preserve the symmetric monoidal structure strictly.
A model category structure on $\PCat$ was described in \cite[Thm. 3.1]{Sharma}. This model category structure can be obtained by transfering the natural model category structure on $\Cat$ to $\PCat$ and therefore it is aptly called the \emph{natural} model category structure of permutative categories. In this sectiopn we will describe two new model category structures on $\PCat$ which can be described as model category of permutative groupoids and the model category of (permutative) Picard groupoids.
\subsection{The model category structure of Permutative groupoids}
In this subsection we will construct the desired model category structure of permutative groupoids on $\PCat$. We will do so by transferring the model category structure of groupoids on $\Cat$ along the adjunction
\begin{equation}
\label{Free-pCat-adj}
\F:\Cat \rightleftharpoons \PCat:i
\end{equation}
where $i$ is the forgetful functor and $F$ is its left adjoint namely the \emph{free permutative category} functor. The following lemma will be useful in the proposed construction:

\begin{lem}
	\label{Pi1-perm-gpd}
	The fundamental groupoid of a permutative category is a permutative groupoid.
\end{lem}
\begin{proof}
	Let $C$ be a permutative category and let $-\otimes-:C \times C \to C$ be bifunctor giving the permutative structure.
	From proposition \ref{Pi1-pres-prod}, we have the isomorphism $\Pi_1(C \times C) \cong \Pi_1(C) \times \Pi_1(C)$. Since $\Pi_1(C)$ is a groupoid, the universal property of $\Pi_1(C \times C)$ and the above isomorphism imply that we have a dotted arrow in the following diagram:
	\begin{equation*}
	\xymatrix{
		C \times C \ar[r]^{-\otimes-} \ar[d] & C \ar[d] \\
		\Pi_1(C) \times \Pi_1(C) \ar@{-->}[r] & \Pi_1(C)
	}
	\end{equation*}
	
	which makes the entire diadram commutative. This bifunctor, represented by the dotted arrow in the above diagram, provides a permutative structure on the groupoid $\Pi_1(C)$.
	The symmetry natural transformation of $C$ is a functor
	\begin{equation*}
	\gamma_C:C \times C \times J \to C
	\end{equation*}
	Once again by proposition \ref{Pi1-pres-prod} the free groupoid generated by $C \times C \times J$ is $\Pi_1(C) \times \Pi_1(C) \times J$. Again the universal property of $\Pi_1(C \times C \times J)$ and the above isomorphism imply that we have a dotted arrow in the following diagram:
	\begin{equation*}
	\xymatrix{
		C \times C \times J \ar[r]^{\gamma_C} \ar[d] & C \ar[d] \\
		\Pi_1(C) \times \Pi_1(C) \times J \ar@{-->}[r]_{\ \ \ \ \ \  \ \ \ \gamma_{\Pi_1(C)}} & \Pi_1(C)
	}
	\end{equation*}
	which is the symmetry natural isomorphism of $\Pi_1(C)$.
\end{proof}
\begin{rem}
	\label{symm-mon-Pi1}
	The functor $\Pi_1$ restricts to a functor on $\PCat$ such that the following diagram commutes:
	\begin{equation*}
	\xymatrix{
	\Cat \ar[r]^{\Pi_1} & \gpd  \\
	\PCat \ar[u]^i \ar[r]_{\Pi_1} & \PGpd \ar[u]_i
    }
	\end{equation*}
	where $\PGpd$ denotes the category of permutative groupoids \emph{i.e.} the full subcategory of $\PCat$ having objects those permutative categories whose underlying categories are groupoids.
	\end{rem}

Now we state the main theorem of this section:
\begin{thm} 
	\label{model-str-Perm-Gpds}
	There is a model category structure on the category of all small
	permutative categories and strict symmetric monoidal functors $\PCat$ in which
	\begin{enumerate}
		\item A fibration is a strict symmetric monoidal functor whose underlying functor is a path fibration of (ordinary) categories and
		\item A weak-equivalence is a strict symmetric monoidal functor whose underlying functor is a groupoidal equivalence of (ordinary) categories.
		\item A cofibration is a strict symmetric monoidal functor having the left lifting property with respect to all maps which are both fibrations and weak equivalences.
	\end{enumerate}
	Further, this model category structure is combinatorial.
\end{thm}
\begin{proof}
	The main tool in proving the above theorem will be Theorem \ref{mdl-str-transfer-tool}. The first condition of this theorem follows from the fact that the adjunction \eqref{Free-pCat-adj} is between locally presentable categories. Now we verify \ref{mdl-str-transfer-tool} $(2)$.  Let $u:C \to D$ be a cofibration which has the left lifting property with respect to all fibrations in $\PCat$. The we get the following commutative diagram in $\PCat$:
	\begin{equation*}
	\xymatrix{
	C \ar[d]_u \ar[r] & \Pi_1(C) \ar[d]_{\Pi_1(u)} \ar[r] & \mathbf{P}(\Pi_1(u)) \ar[d]^{P_{\Pi_1(D)}} \\
	D \ar[r] \ar@{-->}[rru]_L  & \Pi_1(D) \ar@{=}[r] & \Pi_1(D)
    }
	\end{equation*}
	 where $P_{\Pi_1(D)}$ is an isofibration between groupoids and hence a path fibration and $i_{\Pi_1(u) }$ is an acyclic cofibration and hence a groupoidal equivalence from lemma \eqref{fact-acy-cof-perm}. Now by assumption, there exists a dotted lifting arrow which makes the entire diagram commutative. The top and the bottom (composite)  arrows, in the above diagram, are groupoidal equivalences. Now the two out of six property of model categories implies that $u$ is a weak equivalence. The symmetry of the cartesian product shows that if $i$ is an acyclic cofibration, the map $u \Box i$ is an acyclic cofibration.
	\end{proof}
 \begin{nota}
 	We will refer to the above model catgory as the model category of \emph{permutative groupoids} and denote it by $\MdlPCatG$.
 	\end{nota}
  
  The above theorem provides a very clear description of fibrations and weak equivalences in the model category $\MdlPCatG$. The next proposition provides a description of cofibrations:
  \begin{prop}
  	\label{char-cof}
  	The cofibrations in the model category $\MdlPCatG$ are the same as those in the natural model category of permutative categories.
  	\end{prop}
  \begin{proof}
  	It is sufficient to show that the acyclic fibrations are the same in the two model category structures in context. A strict symmetric monoidal functor is an acyclic fibration in $\MdlPCatG$ if it's underlying (ordinary) functor is an acyclic fibration in $\MdlCatG$. However, acyclic fibrations in  $\MdlCatG$ are the same as those in the natural model category structure on $\Cat$.
  	We recall that a strict symmetric monoidal functor is an acyclic fibration in the natural model category structure of permutative categories if it's underlying functor is one in the natural model category structure on $\Cat$.
  	This shows that acyclic fibrations in $\MdlPCatG$ are the same as those in the natural model category of permutative categories.
  	\end{proof}

  In \cite[Appendix A]{Sharma} it was shown that the natural model category structure on $\PCat$ is enriched over the natural model category $\Cat$. We recall that the \emph{tensor product} of this enrichment is the following composite:
  \begin{equation}
  \label{tens-PCat-G}
  -\boxtimes-:\Cat \times \PCat \overset{id \times U} \to \Cat \times \Cat \overset{\times}\to \Cat \overset{\F} \to \PCat
  \end{equation}
  The \emph{cotensor} of this enrichment is given by the bifunctor
  \begin{equation}
  \label{cotens-PCat-G}
  [-,-]:\Cat^{op} \times \PCat \overset{id \times U} \to \Cat \times \Cat \overset{[-,-]}\to \PCat
  \end{equation}
 where $[-,-]$ is the internal Hom of $\Cat$ but it takes values in $\PCat$ if the codomain category is permutative.
 The internal Hom is given by the bifunctor
 \begin{equation}
 \label{int-Hom-PCat-G}
 \StrSMHom{-}{-}:\PCat^{op} \times \PCat \to  \Cat
 \end{equation}
 
 \begin{prop}
 	\label{tens-Quillen-bifunc-PCat-G}
 	The tensor product bifunctor $-\boxtimes-$ is a (left-) Quillen bifunctor with respect to the model category structure of groupoids on $\Cat$ and the model category structure of permutative groupoids on $\PCat$.
 \end{prop}
 \begin{proof}
 	The pentuple $(- \boxtimes -, [-,-], \StrSMHom{-}{-}, \inv{\eta_r}, \eta_l)$ defines an adjunction of two variables. In light of \cite[lemma 4.2.2]{Hovey} it would be sufficient to show that whenever we have a functor $i:W \to X$ which is monic on objects and a strict symmetric monoidal functor $p:Y \to Z$ which is a fibration in $\MdlPCatG$, the following induced functor is a fibration in $\MdlPCatG$ which acyclic whenever $i$ or $p$ is acyclic:
 	\begin{equation*}
 	i \Box p:[X, Y] \to [X, Z] \underset{[W, Z]} \times [W, Y]
 	\end{equation*}
 	The strict symmetric monoidal functor $i \Box p$ is a (acyclic) fibration if and only if the (ordinary) functor 
 	\[
 	U(i \Box p):U([X, Y]) \to U([X, Z] \underset{[W, Z]} \times [W, Y]) = U([X, Z]) \underset{U([W, Z])} \times U([W, Y])
 	\]
 	is a (acyclic) fibration in $\MdlCatG$. Since the model category of groupoids $\MdlCatG$ is enriched over itself, therefore $U(i \Box p)$ is a fibration in $\MdlCatG$ which is acyclic whenever $i$ or $p$ are acyclic because $i$ is a cofibration in $\MdlCatG$ and $p$ is a fibration in $\MdlCatG$.
 \end{proof}
The cartesian closed structure of the model category of groupoids gave us a simplicial structure on the same model category.
Along the lines the above proposition gives us a simplicial model category structure on the model category of permutative groupoids:
\begin{prop}
	\label{simp-mdl-Pcat-G}
	The model category of permutative groupoids is a simplicial model category.
\end{prop}
\begin{proof}
	We will verify the conditions of \cite[Def. 4.1.12]{Hovey} to establish a two varable adjunction between the aforementioned bifunctors. The simplicial Hom bifunctor
	\[
	\MapC{-}{-}{\PCat}:\PCat^{op} \times \PCat \to \sSets
	\]
	is defined to be the composite:
	\begin{equation}
	\label{simp-Hom-PCat-G}
	\PCat^{op} \times \PCat \overset{\StrSMHom{-}{-}} \to \Cat \overset{N} \to \sSets.
	\end{equation}
	This bifunctor is the left hom in the sense of the aforementioned definition. The right hom bifunctor
	\[
	\sSets^{op} \times \PCat \to \PCat
	\]
	is defined as follows:
	\begin{equation}
	\label{simp-cotens-PCat-G}
	\sSets^{op} \times \PCat  \overset{\tau_1^{op} \times id} \to \Cat^{op}  \times \PCat \overset{[-, -]} \to \PCat
	\end{equation}
	The \emph{tensor product} bifunctor
	\[
	\sSets \times \Cat \to \Cat.
	\]
	is defined by the following composite:
	\begin{equation}
	\label{simp-tens-PCat-G}
	 \sSets \times \PCat \overset{\tau_1 \times id} \to \Cat  \times \PCat \overset{- \boxtimes -} \to \PCat .
   \end{equation}
	Let $C$ and $D$ be two permutative categories and let $S$ be a simplicil set. We have the following chain of bijections:
	\[
	\PCat(C, [\tau_1(S), D]) \cong \PCat(C \boxtimes \tau_1(S), D) \cong \Cat(\tau_1(S), [C, D]) \cong \sSets(S, N([C, D])).
	\]
	The above chain of bijections verifies that we have an adjunction of two variables.
	
	Now we need to verify the so called $(SM7)$ axiom. Let $u:C \to D$ be a cofibration in $\PCat$ and $i:S \to T$ be a cofibration in $\sSets$. We want to show that the pushout product:
	\[
	u \Box i:D \boxtimes \tau_1(S) \underset{C \boxtimes \tau_1(S)} \coprod C \boxtimes \tau_1(T) \to D \boxtimes \tau_1(T)
	\]
	is a cofibration in $\PCat$ which is acyclic when either $u$ or $i$ is acyclic.
	The functor $\tau_1$ is a left Quillen functor therefore it preserves cofibrations, therefore $\tau_1(i)$ is a cofibration in $\Cat$ which is acyclic if $i$ is acyclic. Now the result follows from theorem \ref{tens-Quillen-bifunc-PCat-G}.
	which says that $\MdlPCatG$ is a $\MdlCatG$-model category.
	
\end{proof}

\subsection[The model category of Picard groupoids]{The model category of Picard groupoids}
\label{Mdl-Pic}
In this subsection we will construct yet another model category structure on $\PCat$ in which the fibrant objects are Picard groupoids. We obtain the desired model category by carrying out a left Bousfield localization of the model category constructed in the previous subsection, namely $\MdlPCatG$. The model category we construct inherits an enrichment over $\MdlCatG$ and the Kan model category of simplicial sets from its parent model category.

\begin{df}
	A \emph{Picard groupoid} $G$ is a permutative groupoid such that one of the following two functors is an equivalences of categories:
	\begin{equation}
	G \times G \overset{(- \underset{G} \otimes -, p_1)} \to G \times G \ \ \ \ \ \textit{and} \ \ \ \ \ \
	G \times G \overset{(- \underset{G} \otimes -, p_2)} \to G \times G,
	\end{equation}
	where $p_1$ and $p_2$ are the two obvious projection maps.
	\end{df}
\begin{rem}
	If one of the two functors in the above definition is an equivalence of categories then the permutative structure on the groupoid $G$ in the above definition implies the other functor is also an equivalence.
	\end{rem}

\begin{prop}
	A permutative groupoid is a Picard groupoid if and only if for each object $g \in Ob(G)$ there exists another object $\inv{g} \in Ob(G)$ and the following two isomorphisms in $G$:
	\[
	g \underset{G} \otimes \inv{g} \cong \unit{G}, \ \ \ \ \ \ \inv{g} \underset{G} \otimes g \cong \unit{G}
	\]
	\end{prop}
\begin{proof}
	Let us first assume that $G$ is a permutative Picard groupoid, then the functor $G \times G \overset{(- \underset{G} \otimes -, p_1)} \to G \times G$ is an equivalence of categories. This implies that for a pair $(\unit{G}, g)$ of objects of $G$, there exists another pair $(g, \inv{g})$ of objects of $G$, such that
	\[
     (\unit{G}, g) \cong (- \underset{G} \otimes -, p_1)((g, \inv{g})) = (g \underset{G} \otimes \inv{g}, g).
	\]
	This implies that for each $g \in Ob(G)$ there exists another object $\inv{g} \in Ob(G)$ such that $g \underset{G} \otimes \inv{g} \cong \unit{G}$. The secong isomorphism follows similarly.
	
	Conversely, let us assume that inverses exist upto isomorphism. Under this assumption we can construct an inverse functor to $(- \underset{G} \otimes -, p_1)$ as follows:
	\[
	\inv{(- \underset{G} \otimes -, p_1)} (g_1, g_2) := (g_2, \inv{g_2} \underset{G} \otimes g_1).
	\]
	Similarly one can construct an inverse to the second functor.
	
	\end{proof}

We recall the construction of the permutative categories $\L(1)$ and $\L(2)$ from \cite{Sharma}. The permutative category $\L(1)$ is a groupoid whose object set consists of all finite sequences $(s_1, s_2, \dots, s_r)$, where either $s_i = {1}$ or $s_i = 0$ for all $1 \le i \le r$. For an object $S = (s_1, s_2, \dots, s_r)$ in $\L(1)$ we denote by $\ud{S}$ the sum $\underset{i=1}{\overset{r} +} s_i$.
A map $S = (s_1, s_2, \dots, s_r) \to T=(t_1, t_2, \dots, t_k)$ in $\L(1)$ is a bijection $f:\ud{S} \to \ud{T}$.
\begin{prop}
	\label{free-pCat-one-gen}
	For any permutative groupoid $G$, the evaluation map
	\[
	ev_{(1)}:\StrSMHom{\L(1)}{G} \to G
	\]
	is an equivalence of categories.
	\end{prop}
\begin{proof}
	The free permutative category $\F(1)$, see \eqref{Free-pCat-adj}, can be described as follow: The objects are finte sets $\ud{n}$ for all $n \ge 0$. A morphism is a bijection. This category has the property that the evaluation functor on the object $\ud{1}$:
	\[
	ev_{\ud{1}}:\StrSMHom{\F(\ud{1})}{C} \to C
	\]
	is an isomorphism for any permutative category $C$. This category is equipped with an inclusion functor
	\[
	i:\F(\ud{1}) \to \PStr(1),
	\]
	such that $i(\ud{1}) = (1)$, which is an equivalence of categories. Now the $2$ out of $3$ and the following commutative diagram proove the proposition:
	\[
	\xymatrix{
	\StrSMHom{\L(1)}{G} \ar[r]^{ \ \ \ \ ev_{(1)}} \ar[d]_{\StrSMHom{i}{G}} & G \\
	\StrSMHom{\F(\ud{1})}{G} \ar[ru]_\cong
     }
	\]

	\end{proof}
The maps of finite sets $m_2:2^+ \to 1^+$, $\delta^2_1:2^+ \to 1^+$ and $\delta^2_2:2^+ \to 1^+$ together induce the following two maps in $\PCat$
\begin{equation}
\L(1) \vee \L(1) \overset{(\L(m_2), \L(\delta^2_1))} \to \L(2) \ \ \ \ \textit{and} \ \ \ \  \L(1) \vee \L(1) \overset{(\L(\delta^2_1), \L(\delta^2_2))} \to \L(2) 
\end{equation}
By \cite[]{Sharma} the strict symmetric monoidal functor $(\L(\delta^2_1), \L(\delta^2_2))$ is an acyclic cofibration in the natural model category structure on $\PCat$. This implies that for any permutative category $C$, we have the following equivalence of categories:
\begin{equation}
\label{Two-partition-equiv}
\StrSMHom{(\L(\delta^2_1), \L(\delta^2_2))}{C}:\StrSMHom{\L(2)}{C} \to \StrSMHom{\L(1)}{C} \times \StrSMHom{\L(1)}{C}.
\end{equation}
\begin{lem}
	\label{Pic-gpd-as-loc-ob}
	A groupoid $G$ is a $\lbrace (\L(m_2), \L(\delta^2_1)) \rbrace$-local object if and only if it is a picard groupoid.
	\end{lem}
\begin{proof}
	The groupoid $G$ is $\lbrace (\L(m_2), \L(\delta^2_1)) \rbrace$-local if and only if we have the following weak homotopy equivalence of simplicial sets:
	\[
	\Map^h((\L(m_2), \L(\delta^2_1)), G):\Map^h(\L(2), G) \to \Map^h(\L(1) \vee \L(1), G)
	\]
	We recall that the function complex bifunctor for $\MdlPCatG$ is defined as follows:
	\[
	\Map^h(-, -) := N(\StrSMHom{-}{-})
	\]
 which implies that $\Map^h((\L(m_2), \L(\delta^2_1)), G)$ is a homotopy equivalence if and only if the functor:
 \[
 \StrSMHom{(\L(m_2), \L(\delta^2_1))}{G} :\StrSMHom{\L(2)}{G} \to \StrSMHom{\L(1) \vee \L(1)}{G} \cong \StrSMHom{\L(1)}{G} \times \StrSMHom{\L(1)}{G}.
 \]
 is an equivalence of categories. Thus we get the following (composite) weak equivalence in $\MdlPCatG$:
  \begin{equation}
 \StrSMHom{\L(2)}{G} \overset{p} \to \StrSMHom{\L(1)}{G} \times \StrSMHom{\L(1)}{G} \overset{(ev_{(1)}, ev_{(1)})}\to G \times G
 \end{equation}
 where $p = \StrSMHom{(\L(m_2), \L(\delta^2_1))}{G}$.
 There is another composite map in $\PCat$ which is the following:
 \begin{equation}
 \StrSMHom{\L(2)}{G} \overset{q} \to \StrSMHom{\L(1)}{G} \times \StrSMHom{\L(1)}{G} \overset{(ev_{(1)}, ev_{(1)})}\to G \times G \overset{r} \to G \times G
 \end{equation}
 where $q = \StrSMHom{(\L(\delta^2_1), \L(\delta^2_2))}{G}$ and the map $r = (-\underset{G} \otimes -, p_2)$. We will now construct a natural isomorphism (in $\Cat$) $H:(ev_{(1)}, ev_{(1)}) \circ p \Rightarrow r \circ (ev_{(1)}, ev_{(1)}) \circ q $ between the above two functors. 
 For each $F \in \StrSMHom{\L(2)}{G}$ let us denote $F((\ud{2}))$ by $g_{12}$. The isomorphism $p_{12}:(\ud{2}) \cong (\lbrace 1 \rbrace, \lbrace 2 \rbrace)$ in $\L(2)$ gives an isomorphism $F(p_{12}):g_{12} \cong g_1 \otimes g_2$, where $g_1 = F((\lbrace 1 \rbrace))$ and $g_2 = F((\lbrace 2 \rbrace))$. We observe that $r \circ (ev_{(1)}, ev_{(1)}) \circ q (F) = (g_1 \otimes g_2, g_1)$ and $ (ev_{(1)}, ev_{(1)}) \circ p(F) = (g_{12}, g_1)$.
 We define $H(F) := F(p_{12})$. Let $\sigma:F \Rightarrow G$ be a (monoidal) natural transformation and denoting $G((\ud{2}))$ by $g'_{12}$, $G(({1}))$ by $g'_{1}$ and $G(({2}))$ by $g'_{2}$ we get an isomorphism $G(p_{12}):g'_{12} \cong g'_1 \otimes g'_2$. The following diagram commutes:
 \begin{equation*}
 \xymatrix{
 g_{12} \ar[r]^{H(F)} \ar[d]_{\sigma((\ud{2}))} & g_1 \otimes g_2 \ar[d]^{{\sigma((\lbrace 1 \rbrace, \lbrace 2 \rbrace))}} \\
 g'_{12} \ar[r]_{H(G)} & g'_1 \otimes g'_2
  }
 \end{equation*}
 because $\sigma$ is a natural isomorphism. Hence we have constructed the desired natural isomorphism $H$. The construction of $H$ implies that the strict symmetric monoidal functor $r \circ (ev_{(1)}, ev_{(1)}) \circ q$ is a groupoidal equivalence if and only if $(ev_{(1)}, ev_{(1)}) \circ p$ is one. We know that the functors $q$ and $(ev_{(1)}, ev_{(1)})$ are both equivalences of categories. Let us assume that $G$ is a aforementioned local object then $(ev_{(1)}, ev_{(1)}) \circ p$ is a groupoidal equivalence and ,by the above argument, so is the composite functor $r \circ (ev_{(1)}, ev_{(1)}) \circ q$. By two out of three property of weak equivalences this implies that $r$ is a weak equivalence which implies that $G$ is a picard groupoid. Conversely, let us assume that $G$ is a picard groupoid in which case $r$ is a groupoidal equivalence which means that both $r \circ (ev_{(1)}, ev_{(1)}) \circ q$ and $(ev_{(1)}, ev_{(1)}) \circ p$ are groupoidal equivalences. Again by the two out of three property, $p$ is a groupoidal equivalence which implies that $G$ is local.
 
	\end{proof}

 \begin{thm}
	\label{Pic-mdl-str}
	There is a combinatorial model category structure on the category of (small) permutative categories $\PCat$ in which a functor $F:A \to B$ is
	\begin{enumerate}
		\item  a cofibration if it is a cofibration in the natural model category structure on $\PCat$.
		\item a weak equivalence if the following functor
		\[
		\StrSMHom{F}{P}:\StrSMHom{B}{P} \to \StrSMHom{A}{P}
		\]
		is an equivalence of categories for each Picard groupoid $P$.
		\item a fibration if it has the right lifting property with respect to the set of maps which are both cofibrations and weak equivalences.
	\end{enumerate}
A permutative category is a fibrant objects of this model category if and only if it ia a Picard groupoid. 
\end{thm}
\begin{proof}
 We will prove this theorem by localizing the model category of permutative groupoids $\MdlPCatG$ with respect to the map 
 \[
 \L(1) \vee \L(1) \overset{(\L(m_2), \L(\delta^2_1))} \to \L(2).
 \]
 The existence of this left Bousfield localization follows from theorem \ref{local-tool}. A left Bousfield localization preserves cofibrations therefore the cofibrations in the new model category are the same as those in $\MdlPCatG$.
 By proposition \ref{char-cof}, the cofibrations in $\MdlPCatG$ are the same as those in the natural model category of permutative categories.
 Lemma \ref{Pic-gpd-as-loc-ob} above tells us that a permutative groupoid is a $\lbrace (\L(m_2), \L(\delta^2_1)) \rbrace$-local object if and only if it is a Picard groupoid. Since $\MdlPCatG$ is a simplicial model category and its function complex is given by:
 \[
 \map^h(-, -) = N(\StrSMHom{-}{-}),
 \]
 A strict symmetric monoidal functor $F:A \to B$ is a $\lbrace (\L(m_2), \L(\delta^2_1)) \rbrace$-local equivalence if the following simplicial map
 \[
 N(\StrSMHom{F}{P}):N(\StrSMHom{B}{P}) \to N(\StrSMHom{A}{P})
 \]
 is a homotopy equivalence of Kan complexes. Since the nerve functor is a homotopy reflection, see lemma \ref{N-htpy-ref}, therefore $N(\StrSMHom{F}{P})$ is a homotopy equivalence if and only if the functor
 \[
\StrSMHom{F}{P}:\StrSMHom{B}{P} \to \StrSMHom{A}{P}
 \]
 is an equivalence of categories (groupoids).

\end{proof}

\begin{nota}
	We will refer to the above model category as the model category of \emph{Picard groupoids}. We denote this model category by $\MdlPCatP$.
\end{nota}
 Adaptations of arguments used in the proof of propositions \ref{tens-Quillen-bifunc-PCat-G} and \ref{simp-mdl-Pcat-G}, to the model category $\MdlPCatP$  prove the following two analogous propositions:
\begin{prop}
	\label{Cat-G-mdl-Pcat-Pic}
	The bifunctors \eqref{tens-PCat-G}, \eqref{cotens-PCat-G} and \eqref{int-Hom-PCat-G} equip the model category of Picard groupoids $\MdlPCatP$ with a $\MdlCatG$-model category structure.
\end{prop}
and
 \begin{prop}
 	\label{simp-G-mdl-PCat-Pic}
 	The bifunctors \eqref{simp-tens-PCat-G}, \eqref{simp-cotens-PCat-G} and \eqref{simp-Hom-PCat-G} equip the model category of Picard groupoids $\MdlPCatP$ with a simplicial model category structure.
 \end{prop}

 \section[The model category of coherently commutatve monoidal groupoids]{The model category of coherently commutatve monoidal groupoids}

A $\gCat$ is a functor from $\gop$ to $\Cat$.
The category of 
functors from $\gop$ to $\Cat$ and natural transformations between them $\CatHom{\gop}{\Cat}{}$ will be denoted by $\gCAT$. We begin by describing a model category structure on
$\gCAT$ which is often refered to as the \emph{projective model category structure}. We will also refer to this model category structure as \emph{strict model category structure of $\gGpds$}. 
\begin{df}
 A morphism $F:X \to Y$ of $\gCats$ is called
 \begin{enumerate}
 \item a \emph{strict groupoidal equivalence} of $\gCats$  if it is a degreewise weak equivalence in the groupoidal model category structure on $\Cat$ \emph{i.e.} $F(n^+):X(n^+) \to Y(n^+)$ is an equivalence of categories.
 
 \item a \emph{strict path fibration}  of $\gCats$ if it is degreewise a fibration in the groupoidal model category structure on $\Cat$ \emph{i.e.} $F(n^+):X(n^+) \to Y(n^+)$ is a path fibration.
 
 \item a \emph{Q-cofibration}  of $\gCats$ if it has the left lifting property with respect to
 all morphisms which are both strict groupoidal equivalence and strict path fibrations of $\gCats$.

  \end{enumerate}
 \end{df}
 In light of the combinatorial (natural) model category structure on $\Cat$, we observe that a map of $\gCats$ $F:X \to Y$ is a strict acyclic fibration of $\gCats$ if and only if it has the right lifting property with respect to all maps in the set
 \begin{equation}
 \label{gen-cof}
 \I = \lbrace \gn{n} \times \partial_0,
 \gn{n} \times \partial_1,  \gn{n} \times \partial_2 \mid \forall n \in Ob(\N) \rbrace.
 \end{equation}
 We further observe that $F$ is a strict fibration if and only it has the right lifting property with respect to all maps in the set
 \begin{equation}
 \label{gen-acyc-cof}
 \J = \lbrace \gn{n} \times i_0, \gn{n} \times i_1 \mid \forall n \in Ob(\N) \rbrace_{}.
 \end{equation}
 where $\partial_0$, $\partial_1$ and $\partial_2$ are the generating cofibrations of both the natural model category structure on $\Cat$ and the model category structure of groupoids on $\Cat$. The maps $i_0:0 \hookrightarrow I$ and $i_1:1 \hookrightarrow I$ are the obvious inclusions into the source and target of $I$.
 \begin{thm}
 \label{str-mdl-cat-gCat}
 Strict groupoidal equivalences, strict path fibrations and Q-cofibrations of $\gCats$ provide the category $\gCAT$ with a combinatorial model category structure.
 \end{thm}
 A proof of this proposition is given in an appendix of \cite{JL}.
 \begin{rem}
 	\label{fib-gpd-mdl-gcat}
 	An $\gCat$ $X$ is fibrant in the strict model category of groupoids if and only if $X(n^+)$ is a groupoid for each $n^+ \in \gop$.
 	\end{rem}
 \begin{rem}
 	\label{acyc-fib-gpd-mdl-gcat}
 	Strict acyclic path fibrations are degreewise acyclic fibrations in $\MdlCatG$ therefore they are the same as strict fibrations in the strict model category structure on $\gCAT$, \cite{Sharma}, namely they are degreewise acyclic fibrations in the natural model category structure on $\Cat$.
 	This implies that the cofibrations in the strict model category of groupoids are the same as those in the strict model category structure on $\gCAT$.
 \end{rem}
 
 To each pair of objects $(X, C) \in Ob(\gCAT) \times Ob(\Cat)$ we can assign a $\gCat$ 
 $\TensP{X}{C}{}$
 which is defined in degree $n$ as follows:
 \[
 (\TensP{X}{C}{})(n^+) :=  X(n^+) \times C,
 \]
 This assignment
 is functorial in both variables and therefore we have a bifunctor
 \[
 \TensP{-}{-}{}:\gCAT \times \Cat \to \gCAT.
 \]
 Now we will define a couple of function objects for the category $\gCAT$.
 The first function object enriches the category $\gCAT$ over
 $\Cat$ \emph{i.e.} there is a bifunctor
 \[
 \MapC{-}{-}{\gCAT}:\gCAT^{op} \times \gCAT \to \Cat
 \]
 which assigns to any pair of objects $(X, Y) \in Ob(\gCAT) \times Ob(\gCAT)$, a category
 $\MapC{X}{Y}{\gCAT}$ whose set of objects is the following
 \[
 Ob(\MapC{X}{Y}{\gCAT}) := Hom_{\gCAT}(X, Y)
 \]
 and the morphism set of this category are defined as follows:
 \[
 Mor(\MapC{X}{Y}{\gCAT}) := Hom_{\gCAT}(X \times I, Y)
 \]
 For any $\gCat$ $X$, the functor $\TensP{X}{-}{}:\Cat \to \gCAT$ is
 left adjoint to the functor $\MapC{X}{-}{\gCAT}:\gCAT \to \Cat$. The counit of this adjunction
 is the evaluation map $ev:\TensP{X}{\MapC{X}{Y}{\gCAT}}{} \to Y$
 and the unit is the obvious functor $C \to \MapC{X}{\TensP{X}{C}{}}{\gCAT}$.
 To any pair of objects $(C, X) \in Ob(\Cat) \times Ob(\gCAT)$ we can assign a $\gCat$ $\bHom{C}{X}{\gCAT}$
 which is defined in degree $n$ as follows:
 \[
 (\bHom{C}{X}{\gCAT})(n^+) := \CatHom{C}{X(n^+)} \ .
 \]
 This assignment
 is functorial in both variable and therefore we have a bifunctor
 \[
 \bHom{-}{-}{\gCAT}:\Cat^{op} \times \gCAT \to \gCAT.
 \]
 For any $\gCat$ $X$, the functor $\bHom{-}{X}{\gCAT}:\Cat \to \gCAT^{op}$ is
 left adjoint to the functor $\MapC{-}{X}{\gCAT}:\gCAT^{op} \to \Cat$. 
 The following proposition summarizes the above discussion.
\begin{prop}
\label{two-var-adj-cat-gcat}
There is an adjunction of two variables
\begin{multline}
\label{two-var-adj-gcat}
(\TensP{-}{-}{}, \bHom{-}{-}{\gCAT}, \MapC{-}{-}{\gCAT}) : \gCAT \times \Cat
\\  \to \gCAT.
\end{multline}

\end{prop}

 \begin{thm}
  \label{enrich-GamCAT-CAT}
  The strict model category of $\gGpds$ $\gCAT$ is a $\MdlCatG$-enriched model category.
 \end{thm}
 \begin{proof}
  We will show that the adjunction of two variables \eqref{two-var-adj-gcat}
  is a Quillen adjunction for the strict model category structure of $\gGpds$
 on $\gCAT$ and $\MdlCatG$.
  In order to do so, we will verify condition
 (2) of Lemma \ref{Q-bifunctor-char}. Let $g:C \to D$ be a cofibration
 in $\Cat$ and let $p:Y \to Z$ be a strict path fibration of $\gCats$,
 we have to show that the induced map
 \[
  \bhom^{\Box}_{\gCAT}(g, p):\bHom{X}{Y}{\gCAT} \to \bHom{D}{Z}{\gCAT}
  \underset{\bHom{C}{Z}{\gCAT}} \times \bHom{C}{Y}{\gCAT}
 \]
 is a path fibration which is acyclic if either of $g$ or $p$ is
 acyclic. It would be sufficient to check that the above morphism is degreewise
 a fibration in $\Cat$, i.e. for all $n^+ \in \gop$, the morphism
 \begin{equation*}
  \bhom^{\Box}_{\gCAT}(g, p)(n^+): [D, Y(n^+)]  \to
  [D, Z(n^+)] \underset{[C, Z(n^+)]} \times  [C, Y(n^+)] ,
 \end{equation*}
 is a path fibration in $\Cat$. This follows from the observations that the functor $p(n^+):Y(n^+) \to Z(n^+)$
 is a path fibration in $\Cat$ and the model category 
 $\MdlCatG$ is a cartesian closed model category whose internal hom is provided by the bifunctor $[-, -]$.
 \end{proof}
The strict model category of $\gGpds$ is also a simplicial model category:
\begin{prop}
	\label{enrich-over-sSets}
	The strict model category of $\gGpds$ is a simplicial model category.
	\end{prop}
\begin{proof}
	 The right Hom bifunctor of the proposed simplicial enrichment
	\[
	\Map({-}{-}):\gCAT^{op} \times \gCAT \to \sSets
	\]
	is defined as follows:
	\[
	\gCAT^{op} \times \gCAT \overset{\MapC{-}{-}{\gCAT}} \to \Cat \overset{N} \to \sSets.
	\]
	The left hom bifunctor
	\[
	\sSets^{op} \times \gCAT \to \gCAT
	\]
	is defined as follows:
	\[
	\sSets^{op} \times \gCAT \overset{\tau_1 \times id} \to \Cat^{op} \times \gCAT \overset{\bHom{-}{-}{\gCAT}} \to \gCAT
	\]
	The tensor product bifunctor
	\[
	\gCAT \times \sSets \to \gCAT
	\]
	is defined as follows:
	\[
	\gCAT \times \sSets \overset{\id \times \tau_1} \to \gCAT \times \Cat \overset{-\otimes-} \to \gCAT.
	\]
	Let $X, Y$ be $\gCats$ and $S$ be a simplicial sets
	We have the following chain of bijections:
	\begin{multline*}
	\gCAT(X, \bHom{\tau_1(S)}{Y}{\gCAT}) \cong \gCAT(X \otimes \tau_1(S), Y) \cong \gCAT(\tau_1(S) \otimes X, Y) \cong \\ \Cat(\tau_1(S), \MapC{X}{Y}{\gCAT}) \cong \sSets(S, N(\MapC{X}{Y}{\gCAT})).
	\end{multline*}
	The above chain of bijections proves a two-variable adjunction. Now we need to verify the so called $(SM7)$ axiom. Let $u:X \to Y$ be a cofibration in $\gCAT$ and $i:S \to T$ be a cofibration in $\sSets$. We want to show that the pushout product:
	\[
	u \Box i:Y \otimes \tau_1(S) \underset{X \otimes \tau_1(S)} \coprod X \otimes \tau_1(T) \to Y \otimes \tau_1(T)
	\]
	is a cofibration which is acyclic when either $u$ or $i$ is acyclic.
	The functor $\tau_1$ is a left Quillen functor therefore it preserves cofibrations thus $\tau_1(i)$ is a cofibration in $\Cat$ which is acyclic if $i$ is acyclic. Now the result follows from theorem \ref{enrich-GamCAT-CAT}.
	which says that $\gCAT$ is a $\MdlCatG$-model category.
	
	\end{proof}

 Let $X$ and $Y$ be two $\gCats$, the \emph{Day convolution product} of $X$ and $Y$ denoted by $X \ast Y$ is defined as follows:
 \begin{equation}
 \label{Day-Con-prod}
 X \ast Y(n^+) := \int^{(k^+, l^+) \in \gop} \gop(k^+\wedge l^+, n^+) \times X(k^+) \times Y(l^+).
 \end{equation}
 Equivalently, one may define the Day convolution product of $X$ and $Y$ as the left Kan extension of their \emph{external tensor product} $X \overline \times Y$ along the smash product functor
 \[
 - \wedge - :\gop \times \gop \to \gop.
 \]
 we recall that the external tensor product $X \overline \times Y$ is a bifunctor
 \begin{equation*}
 X \overline \times Y:\gop \times \gop \to \Cat
 \end{equation*}
 which is defined on objects by 
 \[
 X \overline \times Y(m^+, n^+) = X(m^+) \times Y(n^+).
 \]
\begin{prop}
\label{GCAT-SM}
The category of all $\gCats$ $\gCAT$ is a symmetric monoidal category under the Day convolution product \eqref{Day-Con-prod}.
The unit of the symmetric monoidal structure is the representable $\gCat$ $\gn{1}$.
\end{prop}
Next we define an internal function object of the category $\gCat$
which we will denote by
\begin{equation}
\label{Int-Map-GCAT}
 \MGCat{-}{-}:\gCAT^{op} \times \gCAT \to \gCAT.
 \end{equation}
 Let $X$ and $Y$ be two $\gCats$, we define the $\gCat$ $\MGCat{X}{Y}$ as follows:
 \begin{equation*}
 \MGCat{X}{Y}(n^+) := \MapC{X \ast \gn{n}}{Y}{\gCAT}.
 \end{equation*}
 \begin{prop}
 \label{closed-SM-cat-GCat}
 The category $\gCAT$ is a closed symmetric monoidal category under the Day convolution product. The internal Hom is given by the bifunctor \eqref{Int-Map-GCAT} defined above.
 \end{prop}
 The above proposition implies that for each $n \in \Nat$
 the functor $- \ast \gn{n}:\gCAT \to \gCAT$ has a right adjoint $\MGCat{\gn{n}}{-}:\gCAT \to \gCAT$. It follows from
 \cite[Thm.]{} that the functor $- \ast \gn{n}$ has another right adjoint which we denote by $-(n^+ \wedge -):\gCAT \to \gCAT$. We will denote $-(n^+ \wedge -)(X)$ by $X(n^+ \wedge -)$, where $X$ is a $\gCat$. The $\gCat$ $X(n^+ \wedge -)$ is defined by the following composite:
 \begin{equation}
 \label{defn-X-n-wedge}
  \gop \overset{n^+ \wedge -} \to \gop \overset{X} \to \Cat.
 \end{equation}
 The following proposition sums up this observation:
 \begin{prop}
  \label{rt-adjs-DayCP}
  There is a natural isomorphism
  \[
  \phi: -(n^+ \wedge -) \cong \MGCat{\gn{n}}{-}.
  \]
  In particular,
  for each $\gCat$ $X$ there is an isomorphism of $\gCats$
  \[
  \phi(X):X(n^+ \wedge -) \cong \MGCat{\gn{n}}{X}.
  \]

 \end{prop}

 The next theorem shows that the strict model category $\gCAT$ is compatible with the Day convolution product.
 \begin{thm}
 \label{SM-closed-mdl-str-GCat}
 The strict  groupoidal model category $\gCAT$ is a symmetric monoidal closed model category under the Day convolution product.
 \end{thm}
 \begin{proof}
 Using the adjointness which follows from proposition \ref{closed-SM-cat-GCat} one can show that if a map $f:U \to V$ is a (acyclic) cofibration in the strict model category $\gCAT$ then the induced map
 $f \ast \gn{n}:U \ast \gn{n} \to V \ast \gn{n}$ is also a (acyclic) cofibration in the strict model category for all $n \in \Nat$. By $(3)$ of Lemma \ref{Q-bifunctor-char} it is sufficient to show that whenever
 $f$ is a cofibration and $p:Y \to Z$ is a fibration then the map
 \[
    \MGBoxCat{f}{p}:\MGCat{V}{Y} \to \MGCat{V}{Z} \underset{\MGCat{U}{Z}}\times \MGCat{U}{Y}.
 \]
   is a fibration in $\gCAT$ which is acyclic if either $f$ or $p$ is a
   weak equivalence. The above map is a (acyclic) fibration if and only if the map
\begin{multline*}
 \MGBoxCat{f \ast \gn{n}}{p}(n^+): \MapC{V \ast \gn{n}}{Y}{\gCAT} \to \\
  \MapC{V \ast \gn{n}}{Z}{\gCAT}
  \underset{\MapC{U \ast \gn{n}}{Z}{\gCAT}} \times \MapC{U \ast \gn{n}}{Y}{\gCAT}
\end{multline*}

    is a (acyclic) fibration in $\Cat$ for all $n \in \Nat$. Since $f \ast \gn{n}$ is a cofibration as abserved above, the result follows from theorem \ref{enrich-GamCAT-CAT}.
 \end{proof}
\subsection[The model category of coherently commutative monoidal groupoids]{The model category of coherently commutative monoidal groupoids}
\label{EInf-Cat}
  The objective of this subsection is to construct a new model
  category structure on the category $\gCAT$. This new model
  category is obtained by localizing the strict model category
 defined above and we call it the \emph{ model category of coherently commutative monoidal groupoids}. The aim of this new model structure is to endow its homotopy category with a semi-additive structure. In other words we want this new
 model category structure to have finite \emph{homotopy biproducts}.  We go on further to show that this new model category is symmetric monoidal with respect to
 the \emph{Day convolution product}, see \cite{Day2}.   
 
 We want to construct a left Bousfield localization of
 the strict model category of $\gCats$. For each pair $k^+, l^+ \in \gop$,
 we have the obvious \emph{projection maps} in $\gSC$
 \[
  \delta^{k+l}_k:(k+l)^+ \to k^+ \ \ \ \ and \ \ \ \ \delta^{k+l}_l:(k+l)^+ \to l^+.
 \]
 The maps
 \[
 \gop(\delta^{k+l}_k,-):\Gamma^{k} \to \Gamma^{k+l} \ \ \ \ and \ \ \ \ 
 \gop(\delta^{k+l}_l,-):\Gamma^{l} \to \Gamma^{k+l} 
 \]
 induce a map of $\gSs$ on the coproduct which we denote as follows:
 \[
  h_k^l:\Gamma^l \vee \Gamma^l \to \Gamma^{l+k}.
 \]
 
 We now define the class of
 maps $\E_\infty\S$ in $\gCAT$ with respect to which we
 will localize.
 \begin{equation}
 \label{loc-maps-ccmc}
  \E_\infty\S := \lbrace h_k^l:\Gamma^l \vee \Gamma^l \to \Gamma^{l+k}:
  l, k \in \mathbb{Z}^+ \rbrace
 \end{equation}
%
 
 The following proposition gives a characterization of (fibrant)
 $\E_\infty\S$-\emph{local objects}, see \cite[Defn. 3.1.4]{Hirchhorn}
 \begin{prop}
 A $\gCat$ $X$ is an $\E_\infty\S$-local object if and
 only if $X(n^+)$ is a groupoid for all $n^+ \in \gop$ and the following functor
 \[
 (X(\delta^{k+l}_k), X(\delta^{k+l}_l)):X((k+l)^+) 
 \to X(k^+) \times X(l^+) 
 \]
 is an equivalence of categories for all $k^+, l^+ \in \gop$.
 \end{prop}
 The proof is an easy consequence of the above definition.
 \begin{df}
 A $\EinC$ \ is a strictly fibrant $\EinSLO$.
 \end{df}
 \begin{df}
 An \emph{equivalence of coheretly commutative monoidal groupoids} is an $\E_\infty \S$-local equivalence.
 \end{df}
 The main result of this section is about constructing a
 new model category structure on the category $\gCAT$,
 by localizing the strict model category of $\gGpds$ with respect to
 morphisms in the set $\E_\infty\S$. 
\begin{thm}
 \label{loc-semi-add}
 There is a closed, left proper, combinatorial model category structure on
 the category of $\gCats$, $\gCAT$, in which
 \begin{enumerate}
 \item Cofibrations are 
 Q-cofibrations of $\gCats$.
 \item Weak equivalences are stable equivalences of $\gCats$.
 \end{enumerate}
 \begin{sloppypar}
 An object is fibrant in this model category if and only if it is a
  coherently commutative monoidal groupoid.
 \end{sloppypar}
 \end{thm}
 \begin{proof}
 The strict model category of $\gCats$ is a combinatorial
 model category therefore the existence of the model structure
 follows from theorem \ref{local-tool} stated above. The statement
 characterizing fibrant objects also follows from theorem \ref{local-tool}.
 \end{proof}
 \begin{nota}
 The model category constructed in theorem \ref{loc-semi-add} will
 be called the model category of \emph{coherently commutative monoidal groupoids}.
 \end{nota}
 \begin{sloppypar}
 The rest of this section is devoted to proving that the model
 category of $\EinCs$  is a symmetric monoidal closed model category.
 In order to do so we will need some general results which we
 state and prove now.
 \end{sloppypar}
 
%
 
 The following lemma will be useful in the proof of the main result of this section:
 \begin{lem}
 	\label{func-cmp-fib}
 	For each $Q$-cofibrant $\gCat$ $W$, the mapping object $\MGCat{W}{A}$ is a coherently commutative monoidal groupoid if $A$ is one.
 	\end{lem}
 \begin{proof}
  Since $A$ is a coherently commutative monoidal category \emph{i.e.} a fibrant object in the model category of coherently commutative monoidal categories, the symmetric monoidal closed structure on the aforementioned model category, \cite[]{Sharma}, implies that $\MGCat{W}{A}$ is a coherently commutative monoidal category. Since $W$ is Q-cofibrant by assumption and $A$ is also a fibrant object in the strict model category of $\gGpds$, theorem \ref{enrich-GamCAT-CAT} implies that $\MGCat{W}{A}(n^+)$ is a groupoid  for each $n^+ \in \gop$. Thus we have shown that $\MGCat{W}{A}$ is a coherently commutative monoidal groupoid if $W$ is Q-cofibrant.
 	\end{proof}

Finally we get to the main result of this subsection. All the lemmas proved above will be useful in proving the following theorem:
\begin{thm}
\label{SM-closed-CCMG}
The model category of coherently commutative monoidal groupoids is a symmetric monoidal closed model category under the Day convolution product.
\end{thm}
\begin{proof}
	The generating $Q$-cofibrations are maps between $Q$-cofibrant objects. For a $Q$-cofibrant object $W$ and a coherently commutative monoidal groupoid $A$, the mapping object $\MGCat{W}{A}$ is a coherently commutative monoidal groupoid by lemma \ref{func-cmp-fib}. The strict model category of $\gGpds$ is symmetric monoidal closed under the Day convolution product by Theorem \ref{SM-closed-mdl-str-GCat}. Now Theorem \ref{SM-closed-Func-Mdl} proves the theorem.

\end{proof}

 \section[The model category of coherently commutative Picard groupoids]{Coherently commutative Picard groupoids}
\label{mdl-cat-cc-pic}
 In this subsection we will introduce a notion of a \emph{coherently commutative} Picard groupoid. We will go on to construct another model category structure on $\gCAT$ whose fibrant objects are the aforementioned objects. A prominent result of this section is that this new model category is symmetric monoidal closed under the Day convolution product thereby giving us a tensor product of Picard groupoids.
 The main result of this paper is that there is a Quillen equivalence between the model category of Picard groupoids, which was constructed above, and the proposed model category of coherently commutative Picard groupoids. The Quillen pair which induces this equivalence is the pair $(\PNat, \Kbar)$ defined in \cite[section 6]{Sharma}.
 
  The mode of construction of this new model category will be localization.
 The following two pairs of maps of based sets:
 \[
 m_2:2^+ \to 1^+ \ \ \ \ \textit{and} \ \ \ \ \delta^2_1:2^+ \to 1^+
 \]
 and
 \[
 m_2:2^+ \to 1^+ \ \ \ \ \textit{and} \ \ \ \ \delta^2_2:2^+ \to 1^+
 \]
 induce two maps of $\gCats$
 \[
 \gn{(m_2, \delta^2_1)}{}:\gn{1}{} \vee \gn{1}{} \to \gn{2}{}
 \]
 and
 \[
 \gn{(m_2, \delta^2_2)}{}:\gn{1}{} \vee \gn{1}{} \to \gn{2}{}
 \]
 \begin{nota}
 	We denote the set $\lbrace \gn{(m_2, \delta^2_1)}{}, \gn{(m_2, \delta^2_2)}{} \rbrace$ by $\P_\infty$.
 \end{nota}
\begin{df}
	A \emph{coherently commutative Picard groupoid} is a coherently commutative monoidal groupoid which is also a $\P_\infty$-local object.
	\end{df}
Unravelling the above definition gives us the following characterization of a coherently commutative Picard groupoid:
\begin{prop}
	\label{char-cc-pic-gpd}
	A $\gCat$ $X$ is a coherently commutative Picard groupoid if and only if it satisfies the following three conditions:
	\begin{enumerate}
		\item For each $k^+ \in Ob(\gop)$, $X(k^+)$ is a groupoid.
		\item For each $k^+, l^+ \in Ob(\gop)$
		\[
		(X(\delta^{k+l}_k), X(\delta^{k+l}_l)):X((k+l)^+) \to X(k^+) \times X(l^+)
		\]
		is a groupoidal equivalence.
		\item The following two maps induced by the maps in $\P_\infty$:
		\[
		X((m_2, \delta^2_1)):X(2^+) \to X(1^+) \times X(1^+) \ \  \textit{and} \ \  X((m_2, \delta^2_2)):X(2^+) \to X(1^+) \times X(1^+)
		\]
		are groupoidal equivalences. 
		\end{enumerate}
	\end{prop}
 \begin{df}
 	A stable equivalence of $\gCats$ is a $\P_\infty$-local equivalence.
 	\end{df}
 A left-Bousfield localization with respect to maps in the set $\P_\infty$, see appendix \ref{loc-Mdl-Cats}, gives us the following model category.
\begin{thm}
	\label{Mdl-CC-Pic}
	There is a closed, left proper, combinatorial model category structure on
	the category of $\gCats$, $\gCAT$, in which
	\begin{enumerate}
		\item The class of cofibrations is the same as the class of
		Q-cofibrations of $\gCats$.
		\item The weak equivalences are stable equivalences of $\gCats$.
	\end{enumerate}
	\begin{sloppypar}
		An object is fibrant in this model category if and only if it is a
		coherently commutative Picard groupoid.
	\end{sloppypar}
\end{thm}
 The following lemma will be useful in the proof of the main result of this section:
 \begin{lem}
 	\label{map-obj-pic-gpd}
 	For each $Q$-cofibrant $\gCat$ $W$, the mapping object $\MGCat{W}{A}$ is a coherently commutative Picard groupoid if $A$ is one.
 \end{lem}
 \begin{proof}
 	If $A$ is a coherently commutative Picard groupoid then it is also a fibrant object in the model category of coherently commutative monoidal groupoids, the symmetric monoidal closed structure on the aforementioned model category, \ref{SM-closed-CCMG}, implies that $\MGCat{W}{A}$ is a coherently commutative monoidal groupoid because $W$ is $Q$-cofibrant by assumption. Thus we have verified $(1)$ and $(2)$ in proposition \ref{char-cc-pic-gpd}. In order to verify $(3)$ in the same proposition we need to show that the following two functors are groupoidal equivalences:
 	\begin{multline*}
 	\MapC{W \ast \gn{(m_2, \delta^2_1)}}{A}{\gCAT}:\MapC{W \ast \gn{2}}{A}{\gCAT} \to \\
 	 \MapC{W\ast \gn{1}}{A}{\gCAT} \times \MapC{W\ast \gn{1}}{A}{\gCAT}
 	\end{multline*}
 	and
 	\begin{multline*}
 	\MapC{W \ast \gn{(m_2, \delta^2_2)}}{A}{\gCAT}:\MapC{W \ast \gn{2}}{A}{\gCAT} \to \\
 	\MapC{W\ast \gn{1}}{A}{\gCAT} \times \MapC{W\ast \gn{1}}{A}{\gCAT}
 	\end{multline*}
 	By adjointness, the morphism of $\gCats$ $\MapC{W \ast \gn{(m_2, \delta^2_1)}}{A}{\gCAT}$ is a groupoidal equivalence if and only if its adjunct map
 \begin{multline*}
 \MapC{W} {\MGCat{\gn{(m_2, \delta^2_1)}}{A}}{\gCAT}:\MapC{W} {\MGCat{\gn{2}}{A}}{\gCAT} \to \\
 \MapC{W} {\MGCat{\gn{1}}{A}}{\gCAT} \times \MapC{W} {\MGCat{\gn{1}}{A}}{\gCAT}
 \end{multline*}
 is one. Since $W$ is $Q$-cofibrant, it is sufficient to show that the morphism
 \[
 \MGCat{\gn{(m_2, \delta^2_1)}}{A}: \MGCat{\gn{2}}{A} \to \MGCat{\gn{1}}{A} \times \MGCat{\gn{1}}{A}
 \]
 is a strict equivalence of $\gGpds$. Since the $\gCats$ $\MGCat{\gn{2}}{A}$ and $\MGCat{\gn{1}}{A}$ are both coherently commutative monoidal groupoids therefore the morphism $\MGCat{\gn{(m_2, \delta^2_1)}}{A}$ will be a strict equivalence of $\gGpds$ if and only if $(\MGCat{\gn{(m_2, \delta^2_1)}}{A})(1^+)$ is a groupoidal equivalence. The following commutative diagram :
 \begin{equation*}
 \xymatrix{
  \MGCat{\gn{2}}{A} \ar[r]^{U \ \ \ \ \ \ \ \ \ \ \ \ } \ar[d]_\cong & \MGCat{\gn{1}}{A} \times \MGCat{\gn{1}}{A} \ar[d]^\cong \\
  A(2^+) \ar[r]_{A((m_2, \delta^2_1))} &A(1^+) \times A(1^+)
  }
 \end{equation*}
 where $U = (\MGCat{\gn{(m_2, \delta^2_1)}}{A})(1^+)$, implies that this map is a groupoidal equivalence because $A$ is a coherently commutative Picard groupoid by assumption.
 \end{proof}

\begin{thm}
	\label{SM-closed-CCPG}
	The model category of coherently commutative Picard groupoids is a symmetric monoidal closed model category under the Day convolution product.
\end{thm}
\begin{proof}
	The generating $Q$-cofibrations are maps between $Q$-cofibrant objects. For a $Q$-cofibrant object $W$ and a coherently commutative Picard groupoid $A$, the mapping object $\MGCat{W}{A}$ is a coherently commutative Picard groupoid by lemma \ref{map-obj-pic-gpd}. The strict model category of $\gGpds$ is symmetric monoidal closed under the Day convolution product by theorem \ref{SM-closed-mdl-str-GCat}. Now applying theorem \ref{SM-closed-Func-Mdl} to the strict model category of $\gGpds$ with the set of morphisms $\S = \E_\infty\S \cup \P_\infty$, see \eqref{loc-maps-ccmc}, proves the theorem.
\end{proof}

 \section[The Quillen equivalences]{The Quillen equivalences}
\label{Quill-Eqs}
This section is devoted to proving the main result of this paper namely the classical Segal's nerve functor, see \cite{segal}, \cite{mandell},\cite{mandell2} \cite{Sharma}  induces a Quillen equivalence between the model category of coherently commutative Picard groupoids and the model category of Picard groupoids $\MdlPCatP$. We begin by briefly recalling that the thickened Segal's nerve functor $\Kbar$ constructed in \cite{Sharma}:

\begin{df}
	For each $n \in Ob(\N)$ we will have a permutative groupoid $\PNat(n)$. The objects of this
	groupoid are finite collections of morphisms in $\gop$ having domain $n^+$, in other words the object monoid of the category $\PNat(n)$ is the free monoid generated by the following set
	\[
	Ob(\PNat(n)) := \underset{k \in Ob(\N)} \sqcup \ \gn{n}(k^+).
	\]
	\begin{sloppypar}
		We will denote an object of this groupoid by $(f_1, f_2, \dots, f_r)$. A morphism $(f_1, f_2, \dots, f_r) \to (g_1, g_2, \dots, g_k)$ is an
		isomorphism of finite sets
		\[
		F:\Supp{f_1} \sqcup \Supp{f_2} \sqcup \dots \sqcup \Supp{f_r} \overset{\cong} \to \Supp{g_1} \sqcup \Supp{g_2} \sqcup \dots \sqcup \Supp{g_k}
		\]
		such that the following diagram commutes
	\end{sloppypar}
	\begin{equation*}
	\xymatrix{
		\Supp{f_1} \sqcup \dots \sqcup \Supp{f_r} \ar[rr]^{F} \ar[rd] && \Supp{g_1} \sqcup \dots \sqcup \Supp{g_k} \ar[ld] \\
		&n
	}
	\end{equation*}
	where the diagonal maps are the unique inclusions of the coproducts into $n$.
\end{df}
\begin{rem}
	\label{cont-fun-Lbar}
	The construction above defines a contravariant functor $\PNat(-):\gop \to \PCat$. A map $f:n^+ \to m^+$ in $\gop$ defines a strict symmetric monoidal functor
	$\PNat(f):\PNat(m) \to \PNat(n)$. An object $(f_1, f_2, \dots, f_r) \in \PNat(m)$ is mapped by this functor to $(f_1 \circ f, f_2 \circ f, \dots, f_r \circ f) \in \PNat(n)$.
\end{rem}
 The thickened Segal's nerve functor is now defined in degree $n$ as follows:
 \[
 \Kbar(C)(n^+) := \StrSMHom{\PNat(n)}{C},
 \]
 where $C$ is a permutative category. The functor $\Kbar$ has a left adjoint, denoted $\PNat$, see \cite{Sharma}. The permutative groupoid $\PNat(n)$ constructed above has a full sub-groupoid denoted $\PStr(n)$ which we now recall:
 \begin{df}
 	\label{P-Str-n}
 	For each $n \in \Nat$ we will now define a permutative groupoid $\PStr(n)$. The objects of this
 	groupoid are finite sequences of subsets of $\ud{n}$. We will denote an object of this groupoid by $(S_1, S_2, \dots, S_r)$, where $S_1, \dots S_r$ are subsets of $\ud{n}$. A morphism $(S_1, S_2, \dots, S_r) \to (T_1, T_2, \dots, T_k)$ is an
 	isomorphism of finite sets $F:S_1 \sqcup S_2 \sqcup \dots \sqcup S_r \overset{\cong} \to T_1 \sqcup T_2 \sqcup \dots \sqcup T_k$ such that the following
 	diagram commutes
 	\begin{equation*}
 	\xymatrix{
 		S_1 \sqcup S_2 \sqcup \dots \sqcup S_r \ar[rr]^{F} \ar[rd] && T_1 \sqcup T_2 \sqcup \dots \sqcup T_k \ar[ld] \\
 		&\ud{n}
 	}
 	\end{equation*}
 	where the diagonal maps are the unique inclusions of the coproducts into $\ud{n}$.
 \end{df}
 The (classical) Segal's nerve functor is now defined in degree $n$ as follows:
 \[
 \KSeg(C)(n^+) := \StrSMHom{\PStr(n)}{C},
 \]
 where $C$ is a permutative category. The functor $\KSeg$ has a left adjoint, denoted $\PStr$, see \cite{Sharma}. 
\begin{prop}
	\label{Th-SegN-Pic-to-Pic}
	The classical Segal's nerve functor maps Picard groupoids to coherently commutative Picard groupoids.
	\end{prop}
\begin{proof}
	The following strict symmetric monoidal functor is a weak-equivalence between cofibrant objects  in the model category $\MdlPCatP$
	\[
	(\L(m_2), \L(\delta^2_1)):\PStr(1) \vee \PStr(1) \to \PStr(2).
	\]
	This implies that the following functor is an equivalence of categories:
	\[
	\StrSMHom{(\L(m_2), \L(\delta^2_1))}{C}:\StrSMHom{\PStr(2)}{C}  \to \StrSMHom{\PStr(1)}{C} \times \StrSMHom{\PStr(1)}{C} 
	\]
	This implies that $\KSeg(C)$ is a coherently commutative Picard groupoid. 
	Now the statement about $\Kbar(C)$ follows from the existence of a natural weak-equivalence (homotopy) $\KSeg \Rightarrow \Kbar$ from \cite[Cor. 6.15]{Sharma}.
	\end{proof}

 First we show that the aforementioned adjoint pairs are Quillen pairs between relevant model categories of Picard groupoids constructed in this paper:  
\begin{lem}
	\label{Thick-Qu-pair-Pic}
	The adjoint functors $(\PNat, \Kbar)$ form a Quillen pair between the model category of coherently commutative Picard groupoids and the model category of Picard groupoids.
	\end{lem}
\begin{proof}
	We begin by showing that the adjoint pair $(\PNat, \Kbar)$ is a Quillen pair between the aforementioned model categories. The cofibrations in the model category of coherently commutative monoidal categories are the same as those in the model category of coherently commutative Picard groupoids. The cofibrations in the natural model category of permutative categories are the same as those in the model category of Picard groupoids. It was shown in \cite{Sharma} that $(\PNat, \Kbar)$ is a Quillen pair between the model category of coherently commutative monoidal categories and the natural model category of permutative categories. Therefore by the above observation the left adjoint $\PNat$ maps Q-cofibrations to cofibrations in the model category of Picard groupoids. A fibration between Picard groupoids in the model category of Picard groupoids is the same as a fibration in the natural model category of permutative categories, namely an isofibration of the underlying categories. The right adjoint $\Kbar$ maps these fibrations to fibrations in the model category of coherently commutative Picard groupoids. Thus the left adjoint preserve cofibrations and the right adjoint preserves fibrations between fibrant objects therefore $(\PNat, \Kbar)$ is a Quillen pair between the aforementions model categories.
	\end{proof}

The stricter version of the functor $\Kbar$ namely $\KSeg$ is also a right Quillen functor between the model categories considered in the above lemma:
\begin{lem}
	\label{str-Qu-pair-Pic}
	The adjunction $\left(\PStr, \KSeg \right)$ is a Quillen adjunction between the the model category of coherently commutative Picard groupoids  and the model category of Picard groupoids $\MdlPCatP$.
\end{lem}
\begin{proof}
	We will prove the lemma by showing that the left adjoint $\PStr$ preserves cofibrations and  the right adjoint functor $\KSeg$ preserves fibrations between fibrant objects, see \cite[Prop. E.2.14]{AJ1}. The left adjoint preserves cofibrations because the cofibrations in the model category of coherently commutative Picard groupoids are the same as those in the model category of coherently commutative monoidal categories namely $Q$-cofibrations. Further, the cofibrations in the model category $\MdlPCatP$ are the same as those in the natural model category of permutative categories. We recall from \cite{Sharma} that $\PStr$ is a left Quillen functor with respect to the model category of coherently commutative monoidal categories and the natural model category structure on $\PCat$. This implies that $\PStr$ maps cofibrations in the model category of coherently commutative Picard groupoids to cofibrations in $\MdlPCatP$.
	
	 Let $F:C \to D$ be a fibration between fibrant objects in $\MdlPCatP$. In order to show that $\KSeg(F)$ is a fibration in the model category of coherently commutative Picard groupoids, it would be sufficient to show that $\KSeg(F)(n^+)$ is a fibration in $\MdlCatG$, for all $n^+ \in Ob(\gop)$. For each $n \in \Nat$ the groupoid $\PStr(n)$ is a cofibrant object in $\MdlPCatP$. The model category $\MdlPCatP$ is a $\MdlCatG$-model category whose categorical Hom is given by the functor $\StrSMHom{-}{-}$. This implies that the functor
	\[
	\StrSMHom{\PStr(n)}{F}:\StrSMHom{\PStr(n)}{C} \to \StrSMHom{\PStr(n)}{D}
	\]
	is a fibration in $\MdlCatG$ and it is an acyclic fibration in $\MdlCatG$ whenever $F$ is an acyclic fibration.
\end{proof}
 
 Adaptations of arguments in the proofs of the above lemmas to the model category of coherently commutative monoidal groupoids and the model category of permutative groupoids gives us the following two analogous lemmas:
 \begin{lem}
 	\label{Thick-Qu-pair-G}
 	The adjoint functors $(\PNat, \Kbar)$ form a Quillen pair between the model category of coherently commutative monoidal groupoids and the model category of permutative groupoids.
 \end{lem}
\begin{lem}
	\label{str-Qu-pair-G}
	The adjunction $\left(\PStr, \KSeg \right)$ is a Quillen adjunction between the the model category of coherently commutative monoidal groupoids  and the model category of permutative groupoids $\MdlPCatG$.
\end{lem}

The next lemma will be instrumental in writing the proof of the main result of this section:
\begin{lem}
	\label{PNat-pr-pic}
	The left adjoint functor $\PNat$ map coherently commutative Picard groupoids to Picard groupoids.
\end{lem}
\begin{proof}
	Let $X$ be a coherently commutative Picard groupoid. It follows from \cite[Cor. 6.12]{Sharma} that the unit map $\eta_X:X \to \Kbar(\PNat(X))$ is a strict equivalence of $\gCats$. The unit natural transformation gives us the following commutative diagram in $\Cat$:
	\begin{equation*}
	\xymatrix{
		\StrSMHom{\PNat(2)}{\PNat(X)} \ar[r]^{H \ \  \ \ \ \ \ \ \ \ \ }  & \StrSMHom{\PNat(1)}{\PNat(X)} \times \StrSMHom{\PNat(1)}{\PNat(X)}  \\
		X(2^+) \ar[r]_{X((m_2, \partition{2}{1})) \ \ \ \ } \ar[u]^{\eta_X(2^+)} & X(1^+) \times X(1^+) \ar[u]_{\eta_X(1^+) \times \eta_X(1^+)}
	}
	\end{equation*}
	where the top map $H = \StrSMHom{(\PStr(m_2), \PStr(\partition{2}{1}))}{\PNat(X)} = \Kbar(\PNat(X))((m_2, \partition{2}{1}))$. Since the bottom horizontal map and the two vertical maps are equivalences of categories therefore the $2$ out of $3$ property implies that the top horizontal map $H$ is also an equivalence of categories. This implies that $\PNat(X)$ is a Picard groupoid.
\end{proof}

The following lemma will be used in the proof of the main result of this section. This lemma exhibits a property of the Quillen pair in context which is not common to all Quillen equivalences.

\begin{lem}
	\label{char-CC-Mon-eq}
	A morphism of $\gCats$ $F:X \to Y$ is a stable equivalence of $\gCats$ if and only if the strict symmetric monoidal functor
	$\PNat(F):\PNat(X) \to \PNat(Y)$ is a weak equivalence in $\MdlPCatP$.
\end{lem}
\begin{proof}
	Let us first assume that the morphism of $\gCats$ $F$ is a
	stable equivalence of $\gCats$.
	Any choice of a cofibrant replacement functor $Q$ for $\gCAT$ provides a commutative
	diagram
	\begin{equation*}
	\xymatrix{
		Q(X) \ar[r]^{Q(F)} \ar[d] &Q(Y) \ar[d] \\
		X \ar[r]_F & Y
	}
	\end{equation*}
	The vertical maps in this diagram are acyclic fibrations in the model
	category of coherently commutative Picard groupoids which are
	strict equivalences of $\gCats$.
	Applying the functor $\PNat$ to this commutative diagram we get the following
	commutative diagram in $\MdlPCatG$
	\begin{equation*}
	\xymatrix{
		\PNat(Q(X)) \ar[r]^{\PNat(Q(F))} \ar[d] & \PNat(Q(Y)) \ar[d] \\
		\PNat(X) \ar[r]_{\PNat(F)} & \PNat(Y)
	}
	\end{equation*}
	The functor $\PNat$ is a left Quillen functor therefore it preserves
	weak equivalences between cofibrant objects. This implies that the top
	horizontal arrow in the above diagram is a weak equivalence in $\MdlPCatP$.
    \cite[Lemma 6.6]{Sharma} implies that the vertical maps in the
	above diagram are weak equivalences in $\MdlPCatP$. Now the two out of three
	property of weak equivalences in model categories implies that
	$\PNat(F)$ is a weak equivalence in $\MdlPCatP$.
	
	Conversely, let us assume that $\PNat(F):\PNat(X) \to \PNat(Y)$ is a weak equivalence in $\MdlPCatP$. We begin by looking at the special case of $X$ and $Y$ being coherently commutative Picard groupoids. By lemma \ref{PNat-pr-pic}, both $\PNat(X)$ and $\PNat(Y)$ are Picard groupoids therefore $\PNat(F)$ is a weak-equivalence in the natural model category $\PCat$.
		Now \cite[Thm. 6.13]{Sharma} implies that $F$ is a weak-equivalence in the model category of coherently commutative monoidal categories which is also a stable equivalence of $\gCats$.
   Now we tackle the general case. Let $F:X \to Y$ be a morphism of $\gCats$ such that $\PNat(F)$ is
	a weak-equivalence in $\MdlPCatP$. By a choice of a fibrant replacement functor $R$
	we get the following commutative diagram whose vertical arrows are acyclic
	cofibrations in the model category of coherently commutative monoidal categories and $R(X)$ and $R(Y)$ are coherently commutative monoidal categories:
	\begin{equation*}
	\xymatrix{
		R(X) \ar[r]^{R(F)} & R(Y) \\
		X \ar[u]^{\zeta(X)} \ar[r]_F & Y \ar[u]_{\zeta(Y)}
	}
	\end{equation*}
	Applying the functor $\PNat$ to the above diagram we get the following commutative diagram in $\MdlPCatP$:
	\begin{equation*}
	\xymatrix{
		\PNat(R(X)) \ar[r]^{\PNat(R(F))} & \PNat(R(Y)) \\
		\PNat(X) \ar[u]^{\PNat(\zeta(X))} \ar[r]_{\PNat(F)} & \PNat(Y) \ar[u]_{\PNat(\zeta(Y))}
	}
	\end{equation*}
	Since $\PNat$ is a left Quillen functor therefore it preserves
	acyclic cofibrations. This implies that the two vertical morphisms in the above
	diagram are weak-equivalences in $\MdlPCatP$. By assumption $\PNat(F)$
	is a weak-equivalences in $\MdlPCatP$ therefore the two out of three property implies that $\PNat(R(F))$ is a weak-equivalences in $\MdlPCatP$. The discussion earlier in this proof regarding strict equivalence between coherently commutative Picard groupoids implies that $R(F)$ is a strict equivalence of $\gCats$.
\end{proof}

\begin{coro}
	\label{hmtpy-func-K}
	A strict symmetric monoidal functor is a weak equivalences in $\MdlPCatP$ if and only if its image under the right Quillen functor $\Kbar$ is a stable equivalences of $\gCats$.
	\end{coro}
\begin{proof}
	Let $F:C \to D$ be a weak equivalence in $\MdlPCatP$.  We have the following commutative diagram in $\PCat$:
	\begin{equation*}
	\xymatrix{
	\PNat(\Kbar(C)) \ar[r]^{\PNat(\Kbar(F)) \ } \ar[d]_\epsilon  & \PNat(\Kbar(D)) \ar[d]^\epsilon \\
	C \ar[r]_F & D
   }
	\end{equation*}
    \cite[Theorem 6.14]{Sharma} implies that the vertical (counit) maps are weak equivalences in the natural model category $\PCat$ and therefore they are weak-equivalences in $\MdlPCatP$. By assumption $F$ is a weak equivalence therefore by the $2$ out of $3$ property the top horizontal map $\PNat(\Kbar(F))$ is a weak equivalence in $\MdlPCatP$. Now the above lemma implies that $\Kbar(F)$ is a stable equivalence of $\gCats$.
    
    Conversely let us assume that $\Kbar(F)$ is a stable equivalence of $\gCats$. Then the avove lemma implies that $\PNat(\Kbar(F))$ \emph{i.e.} the top horizontal arrow in the above commutative diagram, is a weak equivalence in $\MdlPCatP$. Now the 2 out of 3 property tells us that $F$ is a weak equivalence in $\MdlPCatP$.
	\end{proof}

%
Our next goal is to show that all of the aforementioned Quillen adjunctions are Quillen equivalences.

\begin{thm}
	\label{Thick-Qu-Eq}
	The adjunction $(\PNat, \Kbar)$ is a Quillen equivalence between the model category of coherently commutative Picard groupoids and the model category of Picard groupoids $\MdlPCatP$.
\end{thm}
\begin{proof}

	Let $X$ be a cofibrant object in the model category of coherently commutative Picard groupoids and let $C$ be a Picard groupoid.
	We will show that a map $F:\PNat(X) \to C$ is a stable equivalence of $\gCats$ if and only if its adjunct map $\phi(F):X \to \Kbar{C}$
	is a weak equivalence in $\MdlPCatP$. Let us first assume that $F$ is an equivalence in $\MdlPCatP$. The adjunct map $\phi(F)$ is defined by the following commutative diagram:
	\begin{equation*}
	\xymatrix{
		\Kbar(\PNat(X)) \ar[r]^{\Kbar(F)} & \Kbar(C) \\
		X \ar[u]^\eta \ar[ru]_{\phi(F)}
	}
	\end{equation*}
	The right adjoint functor $\Kbar$ preserves weak equivalences therefore the top horizontal arrow is a stable equivalence of $\gCats$.
	The unit map $\eta$ is a coherently commutative monoidal equivalence by  \cite[Cor. 6.12]{Sharma} and therefore it is a stable equivalence of $\gCats$. Now the 2 out of 3 property of model categories implies that $\phi(F)$ is also a coherently commutative monoidal equivalence. 
	
	Conversely, let us assume that $\phi(F)$ is a stable equivalence of $\gCats$. The 2 out of 3 property of model categories implies that top horizontal arrow in the above commutative diagram, namely $\Kbar(F)$ is a stable equivalence of $\gCats$. Now the above corollary tells us that $F$ is a weak equivalence in $\MdlPCatP$.
	
\end{proof}

The same natural equivalence $\eta:\KSeg \Rightarrow \Kbar$ as in the proof of \cite[Cor. 6.15]{Sharma} implies the proof the following theorem:
\begin{coro}
	\label{main-result}
	The adjunction $(\PStr, \KSeg)$ is a Quillen equivalence between the model category of coherently commutative Picard groupoids and the model category of Picard groupoids $\MdlPCatP$.
\end{coro}

All of the arguments used in the proof of the above corollary can be adapted to the model category pair of the model category of coherently commutative monoidal groupoids and the model category of permutative groupoids to give us the following result:

\begin{thm}
	\label{main-result-G}
	The Quillen pair $(\PStr, \KSeg)$ is a Quillen equivalence between the model category of coherently commutative permutative groupoids and the model category of permutative groupoids $\MdlPCatG$.
\end{thm}

 \section{Stable homotopy one-types}
In this section we will compare our model category of coherently commutative Picard groupoids with a homotopy theory of \emph{stable homotopy one-types}. The main result of this subsection is the existence of an equivalence of categories between the homotopy category of our model category of coherently commutative Picard groupoiuds and a homotopy category of stable homotopy one-types. In this section we will be dealing with the model category of pointed spaces $\pSSetsK$ and we recall that a map in this model category is a weak equivalence if and only if its underlying (unpointed) simplicial map is a weak homotopy equivalence.
\begin{df}
	\label{Sta-homo-OT}
	A stable homotopy one type is a functor $X:\gop \to \pSSets$ such that the following conditions are satisfied:
	\begin{enumerate}
		\item For each $n^+ \in \gop$, the (pointed) simplicial set $X(n^+)$ is a Kan complex.
		\item  The pointed simplicial set $X(1^+)$ has at most two non-trivial homotopy groups only in degrees zero or one.
		\item For each pair of objects $k^+, l^+ \in \gop$, the following simplicial map is a weak homotopy equivalence:
		\[
		(X(\delta^{k+l}_k), X(\delta^{k+l}_l)):X((k+l)^+) \to X(k^+) \times X(l^+)
		\]
		\item The following two maps induced by the maps in $\P_\infty$:
		\[
		X((m_2, \delta^2_1)):X(2^+) \to X(1^+) \times X(1^+) \ \  \textit{and} \ \  X((m_2, \delta^2_2)):X(2^+) \to X(1^+) \times X(1^+)
		\]
		are weak homotopy equivalences of pointed simplicial sets. 
	\end{enumerate}
\end{df}
\begin{rem}
	Each stable homotopy one type is a fibrant object in the stable $Q$-model category constructed in \cite{Schwede}.
	\end{rem}
\begin{rem}
	Each stable homotopy one-type determines a \emph{connective spectrum} with at most two non-trivial homotopy groups in degree zero or one, see \cite{BF78}.
\end{rem}
\begin{rem}
	The adjoint pair of functors $(\tau, N)$ induce an adjunction 
	\begin{equation*}
	[\gop, \tau]:\gCAT \rightleftharpoons \gSC:[\gop, N]
	\end{equation*}
	This adjunction is a Quillen pair with respect to the strict (or projective) model category structires on the two functor categories, see \cite[Remark A.2.8.6]{JL}. Since the counit of $(\tau, N)$ is the identity, therefore the counit of the induces adjunction is also identity.
	\end{rem}
We recall from \cite{sharma4} the adjoint pair $((-)^{\textit{nor}}, U)$ which determines a Quillen equivalence between the $JQ$-model category of $\gSs$ and the $JQ$-model category of normalized $\gSs$. It is easy to see that each coherently commutative monoidal Picard groupoid $X$ determines a $\gS$ upon composition with the nerve functor, we denote this $\gS$ by $N(X)$. Applying the left adjoint gives us a normalized $\gS$ $(N(X))^{\textit{nor}}$. This leads us to the following proposition:
\begin{prop}
	\label{preserve-pic}
	For each coherently commutative Picard groupoid $X$, the normalized $\gS$ $(N(X))^{\textit{nor}}$ is a stable homotopy one-type.
\end{prop}
\begin{proof}
	The nerve functor preserves products and also maps groupoidal equivalences of categories to weak homotopy equivalences of simplicial sets therefore $N(X)$ is a coherently commutative monoidal quasi-category in which $N(X)(k^+)$ is a Kan complex for each $k^+ \in \gop$. Further the simplicial maps:
	\[
	N(X((m_2, \delta^2_1))):N(X(2^+)) \to N(X(1^+)) \times N(X(1^+)) \ \ 
	\]
	and
	\[
	N(X((m_2, \delta^2_2))):N(X(2^+)) \to N(X(1^+)) \times N(X(1^+))
	\]
	are both weak homotopy equivalences of Kan complexes. It follows from \cite[Prop. 6.6]{sharma4} that the unit simplicial map $\eta_{N(X)}:N(X) \to U((N(X))^{\textit{nor}})$ is a strict $JQ$-equivalence of $\gSs$. This implies that the normalized $\gS$ $(N(X))^{\textit{nor}}$ is a stable homotopy one-type.
	
\end{proof}

We recall that a \emph{relative category} $C= (C, W)$ consists of a pair of categories $(C, W)$ which have have the same set of objects and the set arrows of $W$ is a subset of arrows of $C$ and the maps of $W$ are called \emph{weak-equivalences} of $C$. A morphism of relative categories $F:(C, W) \to (D, X)$ is a functor $F:C \to D$ that preserves weak-equivalences. A morphism of relative categories is called a \emph{functor of relative categories}.
\begin{df}
	\label{Hom-Rel-Cat}
	A \emph{strict homotopy} between two functors of relative categories $F:(C, W) \to (D, X)$ and $G:(C, W) \to (D, X)$ is a natural transformation $H:F \Rightarrow G$ such that for each object $c \in C$, the map $H(c)$ lies in $X$, \emph{i.e.}, it is a weak-equivalence in $D$.
	
	More generally, $F$ and $G$ we will say that there exists a homotopy between $F$ and $G$ if they can be joined by a finite zig-zag of strict homotopies.
	\end{df}
Based on a homotopy, there is a notion of homotopy equivalence:
\begin{df}
	\label{Hom-Eq}
	A functor of relative categories $F:(C, W) \to (D, X)$ is called a \emph{strict homotopy equivalence} if there exists another functor of relative categories $\inv{F}:(D, X) \to (C, W)$ and two strict homotopies $\eta:id \Rightarrow \inv{F} \circ F$ and $\epsilon:F \circ \inv{F} \Rightarrow id$.
	
	$F$ will be called a \emph{homotopy equivalence} if $\eta$ and $\epsilon$ are just homotopies, namely, zig-zags of strict homotopies.
	\end{df}
\begin{rem}
	\label{Hom-Eq-Eq-Ho-Cat}
	A homotopy equivalence induces an equivalence on the homotopy categories of its domain and codomain relative categories.
	\end{rem}
Next we will construct three relative categories:
\begin{df}
	\label{Hom-th-Pic}
 We denote by $\PicH$ the relative category in which $\mathbf{Pic}$ is the category whose objects are permutative Picard groupoids and arrows are strict symmetric monoidal functors. The morphisms of $\textit{Str}$ are those strict symmetric monoidal functors whose underlying functors are equivalences of categories.
 \end{df}
 \begin{rem}
 	\label{Ho-eq-Perm-Pic}
 	The homotopy category of the relative category $\PicH$ is equivalent to the homotopy category of the model category $\MdlPCatP$.
 \end{rem}

\begin{df}
	\label{Hom-th-GCat}
	We denote by $\CCPicH$ the relative category in which $\gCAT^f$ is the full subcaregory of $\gCAT$ whose objects are coherently commutative Picard groupoids. The morphisms of $\textit{Str}$ are strict equivalences of $\gCats$.
\end{df}
\begin{rem}
	\label{Ho-eq-CC-Pic}
	The homotopy category of the relative category $\CCPicH$ is equivalent to the homotopy category of the model category of coherently commutative Picard groupoids.
\end{rem}

\begin{df}
	\label{Hom-th-SOT}
	We denote by $\SOTPicH$ the relative category in which $\pGSC^f[1]$ is the full subcategory of $\pGSC$, namely the category of normalized $\gSs$, see \cite{sharma4}, whose objects are stable homotopy one types, see definition \eqref{Sta-homo-OT}. The morphisms of $\textit{Str}$ are strict $JQ$-equivalences of normalized $\gSs$.
\end{df}
\begin{rem}
	\label{Ho-eq-SOT}
	The homotopy category of the relative category $\SOTPicH$ is equivalent to the full subcategory of the homotopy category of the  stable $Q$-model category, constructed in \cite{Schwede}, whose objects are normalized $\gSs$ having at most two non-zero stable homotopy groups only in degree zero or one.
\end{rem}

We recall the classical result that the homotypy theory of one-types \emph{i.e.} Kan complexes (fibrant simplicial sets) having at most two non-zero homotopy groups only in degree zero or one is equivalent to the homotopy theory of groupoids. This result can be expressed by the following (strict) equivalence of relative categories:
\begin{equation}
\label{Hom-hy-one}
\tau_1:(\sSets^1, \textit{WH}) \rightleftharpoons (\gpd, \textit{Eq.}):N
\end{equation}
where $\sSets^1$ denotes the full subcategory of $\sSets$ whose objects are one-types and the maps in $\textit{WH}$ are weak homotopy equivalences. The functors in $\textit{Eq.}$ are equivalences of categories.
\begin{nota}
	We denote by $\norN{-}$ the composite functor
	\[
	\gCAT \overset{N} \to \gSC \overset{(-)^{\textit{nor}}} \to \pGSC
	\]
	where $N$ denotes the functor $[\gop, N]:\gCAT \to \gSC$.
	\end{nota}
Proposition \ref{preserve-pic} above implies that the functor $\norN{-}$ restricts to:
\begin{equation}
\norN{-}:\gCAT^f \to \pGSC^f[1].
\end{equation}
\begin{lem}
	The functor $\norN{-}$ is a homotopy equivalence of relative categories.
	\end{lem}
\begin{proof}
	We begin by observing that the following composite functor:
	\[
	\pGSC \overset{U} \to \gSC \overset{\tau_1} \to \gCAT,
	\]
	which we denote by $\unorT{-}$, restricts to a functor
	\[
	\unorT{-}:\pGSC^f[1] \to \gCAT^f.
	\]
This follows by an argument similar to the one in the proof of Proposition \ref{preserve-pic} based on the fact that $U$ and $\tau_1$ preserve strict $JQ$-equivalences. We claim that this functor $\unorT{-}$ is a homotopy inverse of $\norN{-}$. We observe that the functor $\norN{-}$ is a functor of relative categories because $N = [\gop, N]$ is a right Quillen functor and therefore preserves weak-equivalences (strict equivalences) between fibrant objects. The functor $(-)^{\textit{nor}}$ preserves strict equivalences by  \cite[Prop. 6.2]{sharma4}. Similarly, the functor $\unorT{-}$ preserves strict equivalences because both $U$ and $\tau_1 = [\gop, \tau_1]$ do so.

Next , we will construct a homotopy $\beta^c:id \Rightarrow  \unorT{-} \circ \norN{-} $ with the identity (relative) functor on $\CCPicH$. For each $X \in Ob(\gCAT^f)$, the unit of the Quillen equivalence $((-)^{\textit{nor}}, U)$ provides a strict equivalence of $\gSs$ $\eta_X:N(X) \to U(N(X)^{\textit{nor}})$. Applying the left Quillen functor $\tau_1$, we get a weak equivalence in $\CCPicH$, namely, $\tau_1(\eta_X): X = \tau_1(N(X))  \to \tau(U(N(X)^{\textit{nor}}))$.
We define $\beta^c_X = \tau_1(\eta_X)$. One can easily check that this defines a natural transformation $\beta^c$.
Now we define a (strict) homotopy $\beta^u:id \Rightarrow \norN{-} \circ \unorT{-}$. Let $Y$ be a stable homotopy one type. The the unit map of the Quillen adjunction $(\tau_1, N)$
gives a map $\eta_Y: U(Y) \to N(\tau_1(U(Y)))$. Since $Y$ is a stable homotopy one type therefore this map is a weak homotopy equivalence by \eqref{Hom-hy-one}.
Now applying the functor $(-)^{\textit{nor}}$, we get a weak homotopy equivalence 
 \[
 (\eta_Y)^{\textit{nor}}:Y = (U(Y))^{\textit{nor}} \to N(\tau(U(Y)))^{\textit{nor}}.
 \]
  Now we define $\beta^u_Y = (\epsilon_Y)^{\textit{nor}}$. One can easily check that this defines a natural transformation. Thus we have established a (strict) homotopy equivalence.
	\end{proof}

It follows from corollary \ref{PNat-pr-pic} and corollary \cite[6.12]{Sharma} that the left adjoint functor $\PNat$ restricts to a functor of relative categories
\[
\PNat:\CCPicH \to \PicH.
\]
Further, it follows from proposition \ref{preserve-pic} and theorem \ref{main-result} that the right Quillen functor $\Kbar$  restricts to a functor of relative categories:
\[
\Kbar:\PicH \to \CCPicH.
\] 
This leads us to the final lemma of this section:
\begin{lem}
	\label{Hom-eq-Pic-CCP}
	The pair of functors of relative categories $(\PNat, \Kbar)$ determines a (strict) homotopy equivalence between the relative categories $\CCPicH$ and $\PicH$.
	\end{lem}
\begin{proof}
	For each coherently commutative Picard groupoid $X$, the unit map $\eta_X:X \to \Kbar(\PNat(X))$ is a strict equivalence of $\gCats$ by \cite[Cor. 6.12]{Sharma}. For each Picard groupoid $C$, it follows from theorem \ref{main-result} that the counit map $\epsilon_C:\PNat(\Kbar(C)) \to C$ is an equivalence of categories.
	\end{proof}

%
%
%
%
%
%
%
%
%
%
%
 \appendix

 \section[Some constructions on categories]{Some constructions on categories}
\label{path-object-Cat}
In this appendix we recall the construction of a mapping path object associated to a functor which provides a factorization of the functor into an acyclic cofibration followed by an isofibration.
\begin{prop}
The functor
\[
(\partial_0, \partial_1):[J, C] \to C \times C
\]
is a path fibration and the functor $\sigma:C\to [J,C]$ is an equivalence of categories. Moreover, the functors $\partial_1$ and $\partial_0$ are equivalences surjective on objects.
\end{prop}
\begin{proof}
Let us show that the functor $(\partial_1,\partial_0)$ is an isofibration. Let $ a:A_0\to A_1$ be an object of $[J, C]$ and let $(u_0,u_1):(A_0,A_1)\to (B_0,B_1)$ be an isomorphism in $C \times C$. There is then a unique isomorphism $b:B_0\to B_1$ such that the square
\[
\xymatrix@C=11mm{
A_0  \ar[d]_{u_0} \ar[r]^a  &A_1 \ar[d]^{u_1}   \\
B_0 \ar[r]_{b} &B_1
}
\]
commutes, namely $b = u_1 \circ a \circ \inv{u_0}$. The pair $u=(u_0,u_1)$ defines an isomorphism $a\to b$ in the category $[J, C]$, and we have $(\partial_1,\partial_0)(u)=(u_0,u_1)$. This proves that $(\partial_1,\partial_0)$ is an isofibration. Next we will show that the functor $\sigma$ is an equivalences of categories. We recall from \cite{Sharma} that the natural model category $\Cat$ is cartesian closed and we observe that the inclusion functor $j:0 \hookrightarrow J$ is an equivalence of categories therefore the following chain of maps is an equivalence of categories:
\[
\sigma:C \cong [0, C] \overset{[\inv{j}, C]} \to [J, C]
\]
The functor $\partial_1$ is surjective on objects, since $\partial_1\sigma=id_{C}$. Similarly, the functor $\partial_0$ is surjective on objects. A argument similar to the one above shows that $\partial_0$ and $\partial_1$ are equivalences of categories.
\end{proof}
\begin{df}
The mapping path object associated to a functor $F:X \to Y$ is the category $\POb{F}$ defined by the following pullback square.
\begin{equation*}
\label{mapping-path-object}
\xymatrix@C=11mm{
X \ar@{-->}[rd]_{i_X} \ar@/^/[rrd]^{\sigma F}  \ar@/_/[rdd]_{(id_X, F)} \\
&\POb{F} \ar[d]^{(P_X, P_Y)} \ar[r]^P  &[J, Y] \ar[d]^{ (\partial_0, \partial_1)}   \\
&X \times Y \ar[r]_{F \times id_Y} &Y \times Y
}
\end{equation*}
\end{df}
There is a (unique) functor $ i_{X}:{X} \to \mathbf{P}(F) $ such that $Pi_X= \sigma F$,  $P i_{X} = \sigma F$ and $P_{{X}} i_{{X}}=id_{{X}}$ since square \eqref{mapping-path-object} is cartesian and we have $ \partial_1\sigma F=id_{{Y}} F =F id_{{X}}$. Let us put $P_{{Y}}=\partial_0 P$. Then we have
$F=P_{{Y}} i_{{X}}:{X}\to \mathbf{P}(F)\to {Y}$
since $P_{{Y}} i_{{X}}=\partial_0 P i_{{X}}=\partial_0 \sigma F=id_{{Y}} F =F$. This is the \emph{mapping path factorisation} of the functor $F$ in the category $\PCat$. We now present a concrete construction of the pullback above. An object of $\POb{F}$ is a triple $ (y,A,B)$, where $A$ is an object of $X$, $B$ is an object of $Y$ and $y:F(A)\to B$ is an isomorphism in $Y$. We have $P(y,A,B)=y$, $P_{X}(y,A,B)=A$ and $ P_{Y}(y,A,B)=B$. A morphism $(y,A,B)\to (y',A',B')$ in the category $ \POb{F}$ is a pair of maps $u:A\to A'$ and $v:B\to B'$ such that the  following diagram commutes:
\begin{equation*}
\label{mor-in-mapping-path-object}
\xymatrix@C=11mm{
F(A) \ar[r]^y \ar[d]_{F(u)} &B \ar[d]^{v}   \\
F(A') \ar[r]_{y'} & B'
}
\end{equation*}
\begin{lem}
\label{fact-acy-cof-perm}
The functor $P_{{Y}}$ in the mapping path factorisation
\[
F = P_Y i_X:X \to \mathbf{P}(F) \to Y
\]
is an isofibration and the functor $i_{{X}}$ is an equivalence of categories.
\end{lem}

  \section[Monoidal model categories]{Monoidal model categories}
\label{mon-mdl-cat}
\begin{df}
	\label{Q-adj-2-var}
	Given model categories $\C$, $\D$ and $\E$, an adjunction
	of two variables, $\left(\otimes, \bhom_\C, \map_\C, \phi, \psi \right):
	\C \times \D \to \E$, is called a \emph{Quillen adjunction of two variables}, if, given a
	cofibration $f:U \to V$ in $\C$ and a cofibration $g:W \to X$ in $\D$,
	the induced map
	\[
	f \Box g:(V \otimes W) \underset{U \otimes W} \coprod (U \otimes X) \to V \otimes X
	\]
	is a cofibration in $\E$ that is trivial if either $f$ or $g$ is.
	We will refer to the left adjoint of a Quillen adjunction of two
	variables as a \emph{Quillen bifunctor}.
	
\end{df}
The following lemma provides three equivalent characterizations
of the notion of a Quillen bifunctor. These will be useful in this paper
in establishing enriched model category structures.
\begin{lem}\cite[Lemma 4.2.2]{Hovey}
	\label{Q-bifunctor-char}
	Given model categories $\C$, $\D$ and $\E$, an adjunction
	of two variables, $\left(\otimes, \bhom_\C, \map_\C, \phi, \psi \right):
	\C \times \D \to \E$. Then the following conditions are equivalent:
	\begin{enumerate}
		\item [(1)] $\otimes:\C \times \D \to \E$ is a Quillen bifunctor.
		
		\item[(2)] Given a cofibration $g:W \to X$ in $\D$ and a fibration
		$p:Y \to Z$ in $\E$, the induced map
		\[
		\bhom_\C^{\Box}(g, p):\bhom_\C(X, Y) \to \bhom_\C(X, Z)
		\underset{\bhom_\C(W, Z)}\times \bhom_\C(W, Y)
		\]
		is a fibration in $\C$ that is trivial if either $g$ or $p$ is a
		weak equivalence in their respective model categories.
		
		\item[(3)] Given a cofibration $f:U \to V$ in $\C$ and a fibration
		$p:Y \to Z$ in $\E$, the induced map
		\[
		\map_\C^{\Box}(f, p):\map_\C(V, Y) \to \map_\C(V, Z) \underset{\map_\C(W, Z)}\times \map_\C(W, Y)
		\]
		is a fibration in $\C$ that is trivial if either $f$ or $p$ is a
		weak equivalence in their respective model categories.

	\end{enumerate}
	
\end{lem}

\begin{df}
	\label{enrich-model-cat}
	Let $\bigS$ be a monoidal model category. An \emph{$\bigS$-enriched
		model category} is an $\bigS$ enriched category $\bigA$ equipped with
	a model category structure (on its underlying category) such that
	there is a Quillen adjunction of two variables, see definition
	\ref{Q-adj-2-var}, $\left(\otimes, \bhom_{\bigA}, \map_{\bigA}, \phi, \psi \right):
	\bigA \times \bigS \to \bigA$.
	
\end{df}
  \section[Localization in model categories]{Localization in model categories}
\label{loc-Mdl-Cats}
 We begin by recalling the notion of
a \emph{left Bousfield localization}:

\begin{df}
	Let $\M$ be a model category and let $\S$ be a class of maps in $\M$.
	The left Bousfield localization of $\M$ with respect to $\S$
	is a model category structure $L_\S\M$ on the underlying category of $\M$
	such that
	\begin{enumerate}
		\item The class of cofibrations of $L_\S\M$ is the same as the
		class of cofibrations of $\M$.
		
		\item A map $f:A \to B$ is a weak equivalence in $L_\S\M$ if it is an $\S$-local equivalence,
		namely, for every fibrant $\S$-local object $X$, the induced map on homotopy
		function complexes
		\[
		f^\ast:Map_{\M}^h(B, X) \to Map_{\M}^h(A, X)
		\]
		is a weak homotopy equivalence of simplicial sets. Recall
		that an object $X$ is called fibrant $\S$-local if $X$ is fibrant
		in $\M$ and for every element
		$g:K \to L$ of the set $\S$, the induced map on
		homotopy function complexes
		\[
		g^\ast:Map_{\M}^h(L, X) \to Map_{\M}^h(K, X)
		\]
		is a weak homotopy equivalence of simplicial sets.
		
	\end{enumerate}
	
\end{df}

We recall the following theorem
which will be the main tool in the construction of the
desired model category. This theorem first appeared in an unpublished work \cite{smith}
but a proof was later provided by Barwick in \cite{CB1}.
\begin{thm} \cite[Theorem 2.11]{CB1}
	\label{local-tool}
	If $\M$ is a combinatorial model category and $\S$ is a small
	set of homotopy classes of morphisms of $\M$, the left Bousfield localization $L_\S\M$ of
	$\M$ along any set representing $\S$ exists and satisfies the following conditions.
	\begin{enumerate}
		\item The model category $L_\S\M$ is left proper and combinatorial.
		\item As a category, $L_\S\M$ is simply $\M$.
		\item The cofibrations of $L_\S\M$ are exactly those of $\M$.
		\item The fibrant objects of $L_\S\M$ are the fibrant $\S$-local objects $Z$ of $\M$.
		\item The weak equivalences of $L_\S\M$ are the $\S$-local equivalences.
	\end{enumerate}
\end{thm}
Any functor category $[C, D]$ whose objects are functors between $C$ and $D$ and morphisms are natural transformations carries a symmetric monoidal closed model category structure under the Day convolution product if $D$ is symmetric monoidal closed. Further if $D$ is a combinatorial model category then $[C, D]$ carries a projective model category structure. We will denote the Day convolution product of two functors $F, G \in [C, D]$ by $F \ast G$.
The next theorem provides a condition for a localization to preserves the symmetric monoidal structure in functor model categories:

\begin{thm}
	\label{SM-closed-Func-Mdl}
	Let $\M = [C, D]$ be a functor category where $C$ is a small category and $D$ is a combinatorial symmetric monoidal closed model category. Let us further assume that the projective model category structure on $\M$ is symmetric monoidal closed under the Day convolution product. Let $\S$ be a set of maps in $\M$ whose domains and codomains are projective cofibrant objects of $\M$. Let us denote by $\M_\S$ the model category obtained upon localization, with respect to $\S$, of the projective model structure  on $\M$. If the internal mapping object $\MapC{X}{Y}{\M}$ is an $\S$-local object whenever $X$ is projective cofibrant and $Y$ is an $\S$-local object, then the model category $\M_\S$ is symmetric monoidal closed under the Day convolution product.
\end{thm}
\begin{proof}
	Let $i:U \to V$ be a projective cofibration and $j:Y \to Z$ be another projective cofibration. We will prove the theorem by showing that the following \emph{pushout product} morphism
	\begin{equation*}
	i \Box j:U \ast Z \underset{U \ast Y} \coprod V \ast Y \to V \ast Z 
	\end{equation*}
	
	is a projective cofibration which is also an $\S$-local equivalence whenever either $i$ or $j$ is one.
	We first deal with the case of $i$ being a generating projective cofibration. The assumption of a  symmetric monoidal closed structure on the projective model category $\M$ implies that $i \Box j$ is a projective cofibration and we recall that the cofibrations in $\M_\S$ are exactly projective cofibrations. Let us assume that $j$ is an acyclic  cofibration \emph{i.e.} the projective cofibration $j$ is also an $\S$-local equivalence. We recall that the fibrant objects of $\M_\S$ are exactly $\S$-local objects and fibrations in $\M_\S$ between $\S$-local objects are projective fibrations. According to \cite[Proposition 4.22]{Sharma} the projective cofibration $i \Box j$ is an $\S$-local equivalence if and only if it has the left lifting property with respect to all projective fibrations between $\S$-local objects. Let $p:W \to X$ be a projective fibration between two $\S$-local objects. A (dotted) lifting arrow would exists in the following diagram
	\begin{equation*}
	\xymatrix{
		U \ast Z \underset{U \ast Y} \coprod V \ast Y \ar[r] \ar[d] & W \ar[d]^p \\
		V \ast Z \ar@{..>}[ru] \ar[r] & Y
	}
	\end{equation*}
	if and only if a (dotted) lifting arrow exists in the following adjoint commutative diagram
	\begin{equation*}
	\xymatrix{
		X \ar[r] \ar[d]_j & \MapC{V}{W}{\M} \ar[d]^{(i^*, p^*)} \\
		Y \ar@{..>}[ru] \ar[r] & \MapC{U}{X}{\M} \underset{\MapC{U}{Y}{\M}} \times \MapC{V}{Y}{\M}
	}
	\end{equation*}
	The map $(i^*, p^*)$ is a projective fibration in $\M$ by lemma \ref{Q-bifunctor-char} and the assumption that the projective model category structure on $\M$ is symmetric monoidal closed under the Day convolution product with internal Hom given by $\MapC{-}{-}{\M}$. Further the observation that both $V$ and $U$ are projective cofibrant and the assumption on the internal mapping objects together imply that $(j^*, p^*)$ is a projective fibration between $\S$-local objects and therefore a fibration in the model category $\M_\S$. Since $j$ is an acyclic cofibration by assumption therefore the (dotted) lifting arrow exists in the above diagram. Thus we have shown that if $i$ is a projective cofibration and $j$ is a projective cofibration which is also an $\S$-local equivalence then $i \Box j$ is an acyclic cofibration in the model category $\M_\S$.
	Now we deal with the general case of $i$ being an arbitrary projective cofibration. Consider the following set:
	\begin{equation*}
	\I = \lbrace i:U \to V | \ i \Box j \textit{ \ is an acyclic cofibration in \ }  \M_\S \rbrace
	\end{equation*}
	 We have proved above that the set $\I$ contains all generating projective cofibrations.
	We observe that the set $\I$ is closed under pushouts, transfinite compositions and retracts. Thus $S$ contains all projective cofibrations.
	Thus we have proved that $i \Box j$ is a cofibration which is acyclic if $j$ is acyclic. The same argument as above when applied to the second argument of the Box product (\emph{i.e.} in the variable $j$)
	shows that $i \Box j$ is an acyclic cofibration whenever $i$ is an acyclic cofibration in $\M_\S$.
	
\end{proof}

  \section[Tranfer model structure]{Tranfer model structure on locally presentable categories}
\label{trans-mdl-str-lp}
In this section we will show that one can always transfer a model category structure
 on a locally presentable category $C$ along an adjunction $F:D \rightleftharpoons C:G$, where $D$ is a cofibrantly generated model category.

This result is a special case of the following theorem:

\begin{thm} \cite[Theorem 3.6]{goer-sch}
	\label{mdl-str-transfer-tool}
	Let $F : D \rightleftharpoons C : G$ be an adjoint pair and suppose $D$ is a
	cofibrantly generated model category and $C$ is both complete and cocomplete. Let $I$ and $J$ be chosen sets of generating
	cofibrations and acyclic cofibrations of $D$, respectively. Define a morphism $f : X \to Y$
	in $C$ to be a weak equivalence or a fibration if $G(f)$ is a weak equivalence or fibration
	in $D$. Suppose further that
	\begin{enumerate}
		\item The right adjoint $G : C \to D$ commutes with sequential colimits; and
		\item Every cofibration in $C$ with the LLP with respect to all fibrations is a weak
		equivalence.
	\end{enumerate}
	Then $C$ becomes a cofibrantly generated model category. Furthermore the collections
	$\lbrace F(i) | i \in I \rbrace$ and $\lbrace F(j) | j \in J \rbrace$ generate the cofibrations and the acyclic cofibrations
	of $C$ respectively.
\end{thm}

Now we provide a statement of this special case which will be useful throughout this paper:

\begin{thm}
	\label{trans-mdl-lp-cat}
	Let $F : D \rightleftharpoons C : G$ be an adjoint pair and suppose $D$ is a
	combinatorial model category and $C$ is a locally presentable category. Then there exists a cofibrantly generated model category structure on $C$ in which a map $f$ is
	\begin{enumerate}
		\item a weak equivalence if $G(f)$ is a weak equivalence in $D$
		
		\item a fibration if the map $G(f)$ is a fibration in $D$
		
		\item a cofibration if it has the left lifting property with respect to maps which are both fibrations and weak equivalences.
		
		\end{enumerate}
		 Let $I_D$ and $J_D$ be a chosen class of generating cofibrations and generationg acyclic cofibrations of $D$ respectively. Then the collections $F(I_D) = \lbrace F(i) | i \in I_D \rbrace$ and $F(J_D) = \lbrace F(j) | j \in J_D \rbrace$ generate the cofibrations and the acyclic cofibrations of $C$ respectively.
	\end{thm}

 \bibliographystyle{amsalpha}
 \bibliography{EinGammaCat}

\providecommand{\bysame}{\leavevmode\hbox to3em{\hrulefill}\thinspace}
\providecommand{\MR}{\relax\ifhmode\unskip\space\fi MR }
\providecommand{\MRhref}[2]{%
  \href{http://www.ams.org/mathscinet-getitem?mr=#1}{#2}
}
\providecommand{\href}[2]{#2}
\begin{thebibliography}{dRMMV05}

\bibitem[Bar07]{CB1}
C.~Barwick, \emph{On (enriched) left bousfield localization of model
  categories}, \url{arXiv:0708.2067}, 2007.

\bibitem[BF78]{BF78}
A.~K. Bousfield and E.~M. Friedlander, \emph{Homotopy theory of
  {$\Gamma$}-spaces, spectra and bisimplicial sets.}, Geometric applications of
  homotopy theory II, Lecture Notes in Math. (1978), no.~658.

\bibitem[CMM04]{carrasco-martinez-moreno}
P.~Carrasco and J.~Mart{\'{\i}}nez-Moreno, \emph{Simplicial cohomology with
  coeffiecients in symmetric categorical groups}, Appl.\ Categ.\ Structures
  \textbf{12} (2004), no.~3, 257--285.

\bibitem[Day70]{Day2}
B.~Day, \emph{On closed categories of functors, reports of the midwest category
  seminar {IV}}, Lecture notes in Mathematics, vol. 137, Springer-Verlag, 1970.

\bibitem[dRMMV05]{dr-mm-v}
A.~del R\'io, J.~Mart{\'{\i}}nez-Moreno, and E.~M. Vitale, \emph{Chain
  complexes of symmetric categorical groups}, J. Pure and Appl. Algebra
  \textbf{196} (2005), 279--312.

\bibitem[Dup08]{DP08}
Mathieu Dupont, \emph{Abelian categories in dimension 2}.

\bibitem[EM06]{mandell2}
A.~D. Elmendorf and M.~A. Mandell, \emph{Permutative categories,
  multicategories and algebraic {K}-theory}, Alg. and Geom. Top. \textbf{9}
  (2006), no.~4, 163--228.

\bibitem[FQ93]{fq}
D.~Freed and F.~Quinn, \emph{Chern-{S}imons theory with finite gauge group},
  Commun. Math. Phys. \textbf{156} (1993), no.~3, 435--472.

\bibitem[GK11]{GK11}
Nora Ganter and Mikhail Kapranov, \emph{Symmetric and exterior powers of
  categories}, Transformation Groups \textbf{19} (2011).

\bibitem[GM97]{GM97}
Antonio~R. Garzon and Jesus~G. Miranda, \emph{Homotopy theory for (braided)
  cat-groups}, Cahiers de Topologie et G\'eom\'etrie Diff\'erentielle
  Cat\'egoriques \textbf{38} (1997), no.~2, 99--139 (en). \MR{1454159}

\bibitem[GS07]{goer-sch}
P.~G. Goerss and K.~Schemmerhorn, \emph{Model categories and simplicial
  methods}, Contemp. Math. \textbf{436} (2007), 3--49.

\bibitem[Hir02]{Hirchhorn}
Phillip~S. Hirchhorn, \emph{Model categories and their localizations},
  Mathematical Surveys and Monographs, vol.~99, Amer. Math. Soc., Providence,
  RI, 2002.

\bibitem[Hov99]{Hovey}
M.~Hovey, \emph{Model categories}, Mathematical Surveys and Monographs,
  vol.~63, Amer. Math. Soc., Providence, RI, 1999.

\bibitem[JO12]{JO}
N.~Johnson and A.~M. Osorno, \emph{Modeling stable one-types}, Th. and Appl. of
  Categories \textbf{26} (2012), no.~20, 520–537.

\bibitem[Joy08a]{AJ2}
A.~Joyal, \emph{Notes on quasi-categories},
  \url{http://www.math.uchicago.edu/~may/IMA/Joyal.pdf}, 2008.

\bibitem[Joy08b]{AJ1}
\bysame, \emph{Theory of quasi-categories and applications},
  \url{http://mat.uab.cat/~kock/crm/hocat/advanced-course/Quadern45-2.pdf},
  2008.

\bibitem[JT08]{JT1}
A.~Joyal and M.~Tierney, \emph{Notes on simplicial homotopy theory},
  \url{http://mat.uab.cat/~kock/crm/hocat/advanced-course/Quadern47.pdf}, 2008.

\bibitem[Lur09]{JL}
Jacob Lurie, \emph{Higher topos theory}, Annals of Mathematics Studies, vol.
  170, Princeton University Press, Princeton, NJ, 2009.

\bibitem[Man10]{mandell}
M.~A. Mandell, \emph{An inverse {K}-theory functor}, Doc. Math. \textbf{15}
  (2010), 765--791.

\bibitem[Pat12]{Pat12}
Deepam Patel, \emph{De rham -factors}, Inventiones mathematicae \textbf{190}
  (2012), no.~2, 299--355.

\bibitem[Sch99]{Schwede}
S.~Schwede, \emph{Stable homotopical algebra and {$\Gamma$}- spaces}, Math.
  Proc. Camb. Soc. \textbf{126} (1999), 329.

\bibitem[Sch08]{Schmitt2}
V.~Schmitt, \emph{Enrichments over symmetric picard categories},
  arXiv:0711.0324 (2008).

\bibitem[Seg74]{segal}
G.~Segal, \emph{Categories and cohomology theories}, Topology \textbf{13}
  (1974), 293--312.

\bibitem[Shaa]{sharma4}
A.~Sharma, \emph{The homotopy theory of coherently commutative monoidal
  quasi-categories}, \url{arXiv:1908.05668}.

\bibitem[Shab]{Sharma}
\bysame, \emph{Symmetric monoidal categories and {$\Gamma$}-categories},
  arXiv:1811.11333.

\bibitem[Smi]{smith}
J.~Smith, \emph{Combinatorial model categories}, unpublished.

\bibitem[SV]{SV}
A.~Sharma and A.~A. Voronov, \emph{Categorification of {D}ijkgraaf-{W}itten
  theory}, \url{https://arxiv.org/abs/1511.00295}, Preprint IPMU15-0216.

\end{thebibliography}

\end{document}